\documentclass[a4paper,11pt]{amsart}
\usepackage[left=2.5cm, right=2.5cm,top=2.5cm,bottom=3cm]{geometry}

\usepackage[utf8]{inputenc}
\usepackage[english]{babel}
\usepackage{amssymb}
\usepackage{amsfonts}
\usepackage{dsfont}
\usepackage{comment}

\usepackage{algorithm}
\usepackage{algorithmic}

\usepackage{hyperref}

\usepackage{tikz}
\usepackage{chemarr}

\usepackage{graphicx}
\graphicspath{ {./} }
\usepackage{hhline}
\usepackage[sort&compress,comma,square,numbers]{natbib}

\DeclareMathOperator{\LC}{LC}

\DeclareMathOperator{\SC}{TC}

\DeclareMathOperator{\im}{im}

\DeclareMathOperator{\N}{N}
\DeclareMathOperator{\inte}{int}
\DeclareMathOperator{\relint}{relint}
\DeclareMathOperator{\Conv}{Conv}
\DeclareMathOperator{\Cone}{Cone}

\DeclareMathOperator{\signvar}{\# signvar}
\newcommand{\R}{\mathbb{R}}

\newtheorem{thm}{Theorem}[section] 
\newtheorem{cor}[thm]{Corollary} 
\newtheorem{prop}[thm]{Proposition} 
\newtheorem{lemma}[thm]{Lemma}
\newtheorem{problem}[thm]{Problem}

\theoremstyle{definition}

\newtheorem{ex}[thm]{Example}
\newtheorem{remark}[thm]{Remark}

\begin{document}

\title[The signed support and Descartes' rule of signs]{Geometry of the signed support of a multivariate polynomial and Descartes' rule of signs}

\author{Máté L. Telek}
\address{Department of Mathematical Sciences, University of Copenhagen,
Universitetsparken 5,
2100 Copenhagen, Denmark}
\email{mlt@math.ku.dk}

\begin{abstract}

We investigate the signed support, that is, the set of the exponent vectors and the signs of the coefficients, of a multivariate polynomial $f$. We describe conditions on the signed support ensuring that the semi-algebraic set, denoted as $\{ f < 0 \}$, containing points in the positive real orthant where $f$ takes negative values, has at most one connected component. These results generalize Descartes' rule of signs in the sense that they provide a bound which is independent of the values of the coefficients and the degree of the polynomial. Based on how the exponent vectors lie on the faces of the Newton polytope, we give a recursive algorithm that verifies a sufficient condition for the set $\{ f < 0 \}$ to have one connected component.
We apply the algorithm to reaction networks in order to prove that the parameter region of multistationarity of a ubiquitous network comprising phosphorylation cycles is connected.

\medskip

\end{abstract}

\maketitle


\section{Introduction}

Descartes' rule of signs is a classical theorem in real algebraic geometry that provides an upper bound on the number of positive real roots of a univariate real polynomial. The bound is given by the number of sign changes in the coefficient sequence of the polynomial, therefore it is easy to compute. Since Descartes' bound is independent from the degree of the polynomial, it shows a crucial difference between real and complex roots.

Since Descartes published his result in 1637, a lot of effort has been made to improve and generalize his statement. Gauss showed that the number of positive roots has the same parity as the number of sign changes in the coefficient sequence \cite{Gauss1828}. It is known that the result is valid for polynomials with real exponents \cite{Curtiss1918}. Moreover, Descartes's bound is sharp, that is, for every given sign sequence there exists a polynomial matching the sign sequence that has as many positive roots  as provided by Descartes's bound \cite{Grabiner_DescartesIsSharp}.

The research question of multivariate generalizations is still rather open. In his seminal book \emph{Fewnomials} \cite{khovanskii1991book}, Khovanskii gave an upper bound on the number of positive solutions of a polynomial system given by $n$ real polynomials in $n$ variables that depends only on $n$ and the number of monomials appearing in the polynomials. Khovanskii's bound has been improved in~\cite{bihansottile2007,bihansottile2011}. For some specific systems there are also better bounds available \cite{bihan2007,lirojaswang2003,bihanelhilany2017,koiranportiertavenas2013,avendano2009}. These works generalize Descartes' rule of signs in the sense that they provide upper bounds which are independent of the degree; however, the signs of the coefficients are not taken into account. Recently, in \cite{Bihan_2016,bihan2020optimal}, for systems whose polynomials are supported on a circuit, a sharp upper bound was given that depends on the number of sign changes of a given sequence associated both with the exponents and the coefficients of the polynomials.

Descartes' rule of signs allows also different types of generalizations. In a notorious one, instead of focusing on a system of polynomials one considers a single polynomial in $n$ variables and bounds topological invariants of the hypersurface given by the positive real zero set of the polynomial. In \cite{bihansottile_betti}, the authors provided upper bounds on the sum of the Betti numbers of the hypersurface. Bounds on the connected components of the hypersurface were given in \cite{NewSubexBounds,bihanhumberttavenas2022}. These bounds depend on the number of variables $n$ and the number of monomials of the polynomial.

In this work, we aim to bound connected components but in a slightly different setting. As in \cite{DescartesHypPlane}, we consider connected components of the complement of the hypersurface. To make it precise, let $f\colon \mathbb{R}^{n}_{>0} \to \mathbb{R}$ be a signomial (a generalized polynomial whose exponent vectors are real) and let $f^{-1}(\mathbb{R}_{<0})$ be the set of points in $\mathbb{R}^{n}_{>0}$ where $f$ takes negative values. In~\cite{DescartesHypPlane}, the authors phrased the following problem as \emph{generalization of Descartes' rule of signs to hypersurfaces}.
\begin{problem}
\label{Problem}
Consider a signomial $f\colon \mathbb{R}_{>0}^n \to \mathbb{R}$ with $f(x) = \sum_{\mu \in \sigma(f)}c_{\mu}x^{\mu}$, and $\sigma(f) \subseteq \mathbb{R}^{n}$ a finite set. Find a (sharp) upper bound on the number of connected components of $f^{-1}(\mathbb{R}_{<0})$ based on the signs of the coefficients and the geometry of $\sigma(f)$.
\end{problem}
To avoid wordy sentences, we might call connected components of $f^{-1}(\mathbb{R}_{<0})$ \emph{negative connected components of} $f$. In Section  \ref{Sec_DescBounds}, we give new conditions on the set of exponent vectors $\sigma(f)$ that ensure that $f$ has at most one negative connected component. For instance, the existence of two parallel hyperplanes, which enclose in a certain way the exponent vectors of $f$ with positive coefficients, implies that $f$ has at most one negative connected component  (Theorem \ref{Thm_Box}). In case $f$ is multivariate ($n \geq 2$), we show that the number of negative connected components is one if $f$ has only one positive coefficient (Corollary \ref{Cor_OnePos}), or the exponent vectors are separated by a simplex in a specific manner (Corollary \ref{Lemma_SimplexVertCones}).

In Section \ref{Sec_RedFaces}, we show that the problem of finding the number of negative connected components can be reduced to the same problem for a signomial in fewer monomials if all the exponent vectors of $f$  with negative coefficients are contained in a face of the Newton polytope (Theorem~\ref{Thm_NegFace}). A similar reduction is possible if the Newton polytope of $f$ has two parallel faces containing all the exponent vectors of $f$ (Theorem \ref{Thm_ParallelFacesEdge}). These statements lead to a recursive algorithm that can verify connectivity of $f^{-1} (\mathbb{R}_{< 0})$ (Algorithm \ref{Algo_Conn}). Since the algorithm is based on polyhedral geometry computations, its running time remains reasonable even for polynomials with many variables and many monomials.

\medskip

Connectivity questions concerning semi-algebraic sets naturally appear in the application of algebraic geometry to robotics or to reaction networks \cite{RoboticsCon,ParamGeo}. Our motivation to consider Problem \ref{Problem} came from reaction network theory, where several recent works have focused on the topological properties of subsets of the parameters that give rise to dynamical systems with specific properties. For example, the set of parameters giving rise to toric dynamical systems, known as the toric locus, was shown to be connected in \cite{ConnToric}, and the connectivity of the disguised toric locus has been studied in \cite{Conn_DisgToric}. Questions regarding the connectivity of the parameter region for multistationarity have been investigated in \cite{ParamGeo,MultDualPhos,Multnsite,EF_connpaper}. In~\cite{EF_connpaper}, the authors associated with a reaction network (satisfying some technical conditions) a polynomial function $q\colon \mathbb{R}_{>0}^{n} \to \mathbb{R}$ such that connectivity of $q^{-1}(\mathbb{R}_{<0})$ implies that the parameter region of multistationarity of the reaction network is connected. Using the results from \cite{DescartesHypPlane}, for several reaction networks it was verified that the associated polynomial $q$ has one negative connected component, so the parameter region of multistationarity is connected. 

However, there were some biologically relevant reaction networks where the results from \cite{DescartesHypPlane} did not suffice. In particular, one of these networks was the weakly irreversible phosphorylation system with two binding sites \cite{ParamGeo}. Using numerical methods, the authors in \cite{ParamGeo} showed that its parameter region of multistationarity is connected; however, they did not give a rigorous proof. 

Another important family of reaction networks, where the connectivity of the parameter region of multistationarity has been investigated, is the sequential and distributive phosphorylation cycles  with $m$-binding sites, $m \in \mathbb{N}$. In \cite{EF_connpaper}, it has been showed that the parameter region of multistationarity is connected for $m=2,3$. Furthermore, it is known that the projection of the parameter region of multistationarity to a subset of the parameters (reaction rate constants) is connected for all $m$ \cite{MultDualPhos,Multnsite}. In Section \ref{Sec_App}, we revisit these networks. We use Algorithm \ref{Algo_Conn} to show connectivity of the parameter region of multistationarity for the weakly irreversible phosphorylation system and for phosphorylation cycles  with $m=4,5,6,7$ binding sites. This computational evidence encouraged us to investigate the problem for all $m \in \mathbb{N}$. Since the first submission of the current paper, based on Theorem~\ref{Thm_NegFace} and Theorem~\ref{Thm_ParallelFacesEdge}, in a joint work with Nidhi Kaihnsa we have proved that the parameter region of multistationarity for the sequential and distributive phosphorylation cycles with $m$-binding sites is connected for all $m \in \mathbb{N}$ \cite{NsiteConnecGeneral}.

\subsection*{Notation} $\mathbb{R}_{\geq0}$, $\mathbb{R}_{>0}$ and $\R_{<0}$ refer to the sets of non-negative, positive and negative real numbers respectively. For monomials, we use the notation $x^\mu = x_1^{\mu_1} \cdots  x_n^{\mu_n}$, where $x \in \mathbb{R}_{>0}^n, \mu \in \mathbb{R}^{n}$. For two vectors $v,w \in \mathbb{R}^{n}$, $v \cdot w$ denotes the Euclidean scalar product, and $v \ast w$ denotes the coordinate-wise product of $v$ and $w$. We denote the Euclidean interior of a set $X \subseteq \mathbb{R}^{n}$ by $\inte(X)$. If $X \subseteq \mathbb{R}^{n}$ is a polyhedron, $\relint(X)$ denotes the relative interior of $X$, that is the Euclidean interior of $X$ in its affine hull. The symbol $\#S$ denotes the cardinality of the finite set $S$. We write $ S \sqcup T$ for the disjoint union of two sets $S,T$. 

\section{Separating and enclosing hyperplanes}
\label{Sec_DescBounds}

\subsection{Background and definitions}
\label{Sec::Background}
In this section, we investigate \emph{signomials} and their \emph{negative connected components} using certain affine hyperplanes that partition the exponent vectors of the signomial. Recall that a signomial is a multivariate generalized polynomial with real exponents whose domain is restricted to the positive orthant $\mathbb{R}_{>0}^{n}$ \cite{duffin1973geometric, FuzzyGeo}. In other words, a signomial is a function of the form:
\[ f\colon \, \mathbb{R}^{n}_{>0} \to \mathbb{R}, \quad x \mapsto \sum_{\mu \in \sigma(f)} c_{\mu}x^{\mu},\]
where $\sigma(f) \subseteq \mathbb{R}^n$ is a finite set, called the \emph{support} of $f$ and the \emph{coefficients} $c_\mu$ are non-zero real numbers. The negative connected components of $f$ are connected components of the set:
\[ f^{-1}(\mathbb{R}_{<0}) = \{ x \in \mathbb{R}^{n}_{>0} \mid f(x) < 0 \}.\]
We write $\mathcal{B}^{-}_{0}(f)$ for the set of negative connected components and $b_0( f^{-1}(\mathbb{R}_{<0}))$  for the cardinality of $\mathcal{B}^{-}_{0}(f)$. 
  
 If $c_\mu > 0$ (resp. $c_\mu <0$), we call $\mu$ a \emph{positive} (resp. \emph{negative}) \emph{exponent vector} of $f$. We write
\[ \sigma_+(f) := \{ \mu \in \sigma(f) \mid c_\mu >0\} \quad \text{and} \quad  \sigma_-(f) := \{ \mu \in \sigma(f) \mid c_\mu <0\}\]
for the set of positive and negative exponent vectors respectively. The convex hull of the support
\[ \N(f) := \Conv( \sigma(f) )\]
is called the \emph{Newton polytope} of $f$. For a set $S \subseteq \mathbb{R}^{n}$, we define the \emph{restriction} of $f$ to $S$ as
\[  f_{|S}\colon \mathbb{R}^{n}_{>0} \to \mathbb{R}, \quad x \mapsto  f_{|S}(x) := \sum_{\mu \in \sigma(f) \cap S} c_{\mu}x^{\mu}.\]

Similarly to \cite{DescartesHypPlane}, our arguments will benefit from the existence of \emph{separating} and \emph{enclosing} \emph{hyperplanes} of the support of $f$. We briefly recall these objects. Each $v \in \mathbb{R}^{n} \setminus \{0\}$ and $a \in \mathbb{R}$ define a \emph{hyperplane}
\begin{align*}
\mathcal{H}_{v,a} := \{ \mu \in \mathbb{R}^{n} \mid v \cdot \mu = a \},
\end{align*}
and two \emph{half-spaces}
\begin{align*}
\mathcal{H}^+_{v,a} := \{ \mu \in \mathbb{R}^{n} \mid v \cdot \mu \geq a \}, \qquad \mathcal{H}^-_{v,a} := \{ \mu \in \mathbb{R}^{n} \mid v \cdot \mu \leq a \}.
\end{align*}
The interiors of these half-spaces are denoted by $\mathcal{H}^{+,\circ}_{v,a}$ and $\mathcal{H}^{-,\circ}_{v,a}$. If the support of a signomial $f$ satisfies
\begin{align*}
\sigma_-(f) \subseteq \mathcal{H}^+_{v,a} \qquad \text{and} \qquad \sigma_+(f) \subseteq \mathcal{H}^-_{v,a}
\end{align*}
then we call $\mathcal{H}_{v,a}$ a \emph{separating hyperplane} of $\sigma(f)$ and $v$ is called a \emph{separating vector}.  A separating hyperplane is \emph{strict} if 
\begin{align*}
\sigma_-(f) \cap \mathcal{H}^{+,\circ}_{v,a} \neq \emptyset,
\end{align*}
meaning that the negative exponent vectors are not all contained in the hyperplane.
We call a vector $v \in \mathbb{R}^{n}$ an \emph{enclosing vector} of $\sigma_+(f)$ if there exist parallel hyperplanes $\mathcal{H}_{v,a}$, $\mathcal{H}_{v,b},  \, a \geq b$ such that
\begin{align*}
\sigma_+(f) \subseteq \mathcal{H}^-_{v,a} \cap \mathcal{H}^+_{v,b}, \qquad \text{and} \qquad \sigma_-(f) \subseteq \mathbb{R}^{n} \setminus \big( \mathcal{H}^{-,\circ}_{v,a} \cap \mathcal{H}^{+,\circ}_{v,b} \big).
\end{align*}
In that case, we call the pair $(\mathcal{H}_{v,a},\mathcal{H}_{v,b})$ a \emph{pair of enclosing hyperplanes} of $\sigma_+(f)$. A pair of enclosing hyperplanes is strict if 
\begin{align*}
\sigma_-(f) \cap  \mathcal{H}^{+,\circ}_{v,a} \neq \emptyset \quad \text{and} \quad \sigma_-(f) \cap \mathcal{H}^{-,\circ}_{v,b} \neq \emptyset.
\end{align*}

\begin{ex}
\label{Ex::ExRunning}
To illustrate the above definitions, consider the polynomial 
\begin{align}
\label{Eq_Example2}
f(x,y) = -101  x^{3} y^{2} + 50  x^{2} y^{3} + x y^{3} + y^{4} - x^{2} y - 9.5  y^{3} + 51  x^{2} + 30.5  y^{2} - 37  y + 12.
\end{align}
The set of positive and negative exponent vectors are given by:
\[ \sigma_+(f) = \{(2,3),(1,3),(0,4),(2,0),(0,2),(0,0)\}, \quad \sigma_-(f) = \{ (3,2),(2,1),(0,3),(0,1)\}.\]
The pair  $(\mathcal{H}_{v,2}, \,\mathcal{H}_{v,0}), \, v = (1,0)$ is a pair of enclosing hyperplanes of $\sigma_+(f)$, see Figure~\ref{FIG1}(a).

If we remove the two negative exponent vectors $(0,3),(0,1)$ which are contained in $\mathcal{H}^-_{v,0}$, or in other words we restrict $f$ to $A = (\mathcal{H}_{v,2}^+ \cap \sigma_-(f) )\cup \sigma_+(f)$, then the support of
\begin{align}
\label{Eq_Example1}
f_{|A}(x,y) = -101  x^{3} y^{2} + 50  x^{2} y^{3} + x y^{3} + y^{4} - x^{2} y + 51  x^{2} + 30.5  y^{2} + 12
\end{align}
has a strict separating hyperplane given by $\mathcal{H}_{v,2}, \, v = (1,0)$, see Figure~\ref{FIG1}(b).

By restricting $f$ to $B =  (\mathcal{H}_{v,0}^- \cap \sigma_-(f) )\cup \sigma_+(f)$, we have 
\begin{align}
\label{Eq_Example3}
f_{|B}(x,y) = 50  x^{2} y^{3} + x y^{3} + y^{4} - 9.5  y^{3} + 51  x^{2} + 30.5  y^{2} - 37  y + 12.
\end{align}
The hyperplane $\mathcal{H}_{-v,0}$ is a non-strict separating hyperplane of $\sigma(f_{|B})$, see Figure~\ref{FIG1}(c). The face of $\N(f_{|B})$ with outer normal vector $-v = (1,0)$ equals $\N(f_{|B}) \cap \mathcal{H}_{-v,0}$ and contains all the negative exponent vectors of $f_{|B}$. 
\end{ex}

\begin{figure}[t]
\centering
\begin{minipage}[h]{0.3\textwidth}
\centering
\includegraphics[scale=0.5]{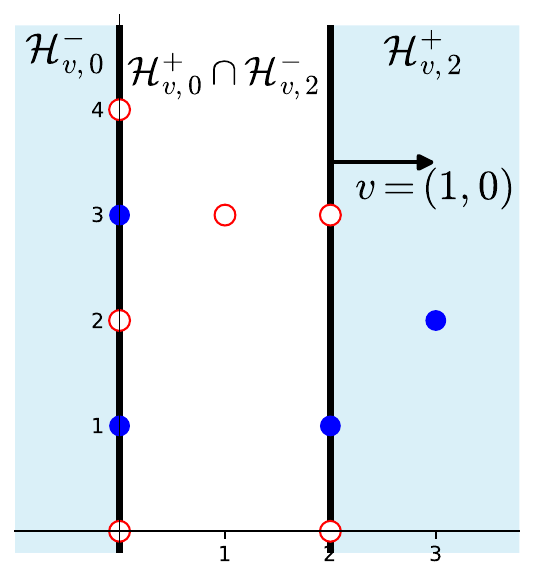}

{\small (a)}
\end{minipage}
\begin{minipage}[h]{0.3\textwidth}
\centering
\includegraphics[scale=0.5]{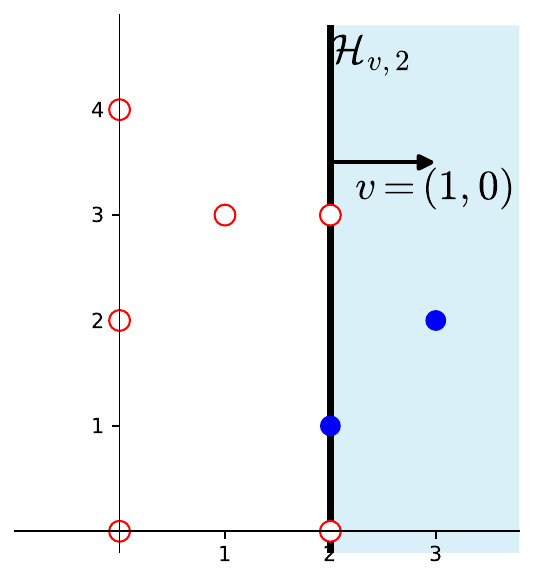}

{\small (b)}
\end{minipage}
\begin{minipage}[h]{0.3\textwidth}
\centering
\includegraphics[scale=0.5]{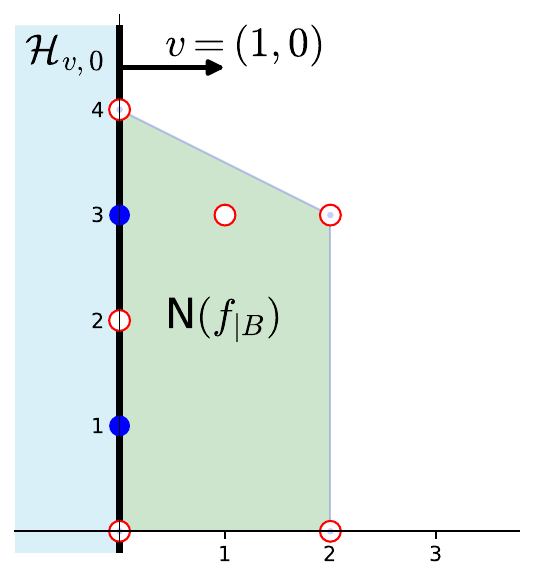}

{\small (c)}
\end{minipage}

\caption{{\small  Exponent vectors of $f,f_{|A},f_{|B}$ from Example \ref{Ex::ExRunning}. Positive exponent vectors are marked by red circles and negative exponent vectors by blue dots. (a) A pair of enclosing hyperplanes of $\sigma_+(f)$.  (b) A strict separating hyperplane of $\sigma(f_{|A})$. (c) The Newton polytope of $f_{|B}$ and a non-strict separating hyperplane of $\sigma(f_{|B})$. }}\label{FIG1}
\end{figure}

The relevance of separating and enclosing vectors arises from the following observation. Each $v \in \mathbb{R}^{n}$ and $x \in \mathbb{R}^{n}_{>0}$ induce a univariate signomial 
\begin{align}
\label{Eq_InducedSignomial} 
\mathbb{R}_{>0} \to \mathbb{R}, \quad t \mapsto f(t^v \ast x) = \sum_{\mu \in \sigma(f)} c_\mu x^\mu t^{v \cdot \mu}.
\end{align}
If $v$ is a separating (resp. enclosing) vector, then $f(t^v \ast x)$, viewed as a signomial in $t$, has  at most one (resp. two) sign changes in its coefficient sequence.

Recall that a coefficient of a univariate signomial $g$ is called the \emph{leading coefficient} $\LC(g)$ (resp. \emph{trailing coefficient} $\SC(g)$) if the corresponding exponent is the largest (resp. smallest).

\begin{ex}
Let $f$ be the signomial in \eqref{Eq_Example2} and let $A,B \subseteq \sigma(f)$ as given in Example \ref{Ex::ExRunning}. For fixed $x \in \mathbb{R}^{n}_{>0}$, the induced univariate signomials as in \eqref{Eq_InducedSignomial} equal
\begin{align*}
 f(t^1 x, t^0 y) &= (-101x^3 y^2)t^3 + (51x^2 -x^2y + 50x^2y^3 )t^2+xy^3 t^1 + y^4 - 9.5y^3+30.5y^2 - 37y +12,\\
f_{|A}(t^1 x, t^0 y) &= (-101x^3 y^2)t^3 + (51x^2 -x^2y + 50x^2y^3 )t^2+xy^3 t^1 + y^4 +30.5y^2  +12,\\
f_{|B}(t^{-1} x, t^0 y) &=     y^4 - 9.5y^3+30.5y^2 - 37y +12   +xy^3 t^{-1} +   (51x^2+ 50x^2y^3 )t^{-2}.
 \end{align*}
 
The leading coefficient of $f_{|A}(t^1 x, t^0 y)$ is $-101x^3 y^2$, which is negative for all $(x,y) \in \mathbb{R}_{>0}^2$. This phenomenon always happens. If $v$ is a strict separating vector of $\sigma(f)$, then the leading coefficient of the induced signomial in \eqref{Eq_InducedSignomial} is negative. On the contrary, this might not be true for non-strict separating vectors. For example, we have
 \begin{align*}
 & \LC(f_{|B}(t^{-1} x, t^0 y)) =  y^4 - 9.5y^3+30.5y^2 - 37y +12 < 0,  & &\text{ if }  y^4 - 9.5y^3+30.5y^2 - 37y +12 < 0, \\
 & \LC(f_{|B}(t^{-1} x, t^0 y))  = xy^3 > 0,  & &\text{ if }  y^4 - 9.5y^3+30.5y^2 - 37y +12 = 0, \\
  & \LC(f_{|B}(t^{-1} x, t^0 y)) =  y^4 - 9.5y^3+30.5y^2 - 37y +12 > 0, & &\text{ if }  y^4 - 9.5y^3+30.5y^2 - 37y +12 > 0.
 \end{align*}
\end{ex}

If $\LC(g) < 0$ (resp. $\SC(g) <0$), then $g$ attains negative values for large (resp. small) enough $t \in \mathbb{R}_{>0}$. This simple observation and the univariate Descartes' rule of signs give the following statement.

\begin{lemma}
\label{Lemma_UniSignchanges}
Let $g\colon \mathbb{R}_{>0} \to \mathbb{R}, \, g(t) = \sum_{i = 1}^{d} a_i t^{\nu_i}$ be a univariate signomial such that $g(1) < 0$.
\begin{itemize}
\item[(i)] If $\LC(g) < 0$ and $g$ has at most one sign change in its coefficient sign sequence, then $g(t) <0$ for all $t\geq 1$.
\item[(ii)] If $\SC(g) < 0$ and $g$ has at most one sign change in its coefficient sign sequence, then $g(t) <0$ for all $t\leq 1$.
\item[(iii)] If $g$ has at most two sign changes in its coefficient sign sequence and $\LC(g) < 0$ or $\SC(g) < 0,$ then $g(t) <0$ for all $t\leq 1$ or $g(t) <0$ for all $t\geq 1$.
\end{itemize}
\end{lemma}

Lemma \ref{Lemma_UniSignchanges} played a crucial role in the proof of the following theorems in \cite{DescartesHypPlane}.

\begin{thm}
\label{Thm::FirstDescartes}
Let $f\colon \mathbb{R}^{n}_{>0} \to \mathbb{R}, \, x \mapsto \sum_{\mu \in \sigma(f)} c_{\mu}x^{\mu}$  be a signomial. 
\begin{itemize}
\item[(i)] If there exists a strict separating hyperplane of $\sigma(f)$, then $f^{-1}( \mathbb{R}_{<0})$ is non-empty and contractible  \cite[Theorem 3.6]{DescartesHypPlane}.
\item[(ii)] If there exists a pair of strict enclosing hyperplanes of $\sigma_+(f)$, then $b_0(f^{-1}( \mathbb{R}_{<0})) \leq 2$ \cite[Theorem 3.8]{DescartesHypPlane}.
\item[(iii)] If $f$ has at most one negative coefficient, then $f^{-1}(\mathbb{R}_{<0})$ is either empty or logarithmically convex \cite[Theorem 3.4]{DescartesHypPlane}.
\end{itemize}
\end{thm}

In addition to Theorem \ref{Thm::FirstDescartes}, another condition on the signed support of the signomial $f$ implying that $f$ has at most one negative connected component, is that the negative and positive exponent vectors of $f$ are separated by a simplex and its negative vertex cones \cite[Theorem 4.6]{DescartesHypPlane}. We postpone recalling this result to Section \ref{Section::OneNegCC} and continue with investigating separating and enclosing hyperplanes.

\subsection{Non-strict separating and non-strict enclosing hyperplanes}
\label{Section::NonStrict}
In the following propositions, we investigate what happens when $\sigma(f)$ (resp. $\sigma_+(f)$) has a separating (resp. enclosing) hyperplane that is not necessarily strict. In these cases, a bound on $b_0(f^{-1}( \mathbb{R}_{<0}))$ is given by the number of negative connected components of restrictions of $f$ to certain subsets of $\sigma(f)$. These technical statements serve as the core part of the proofs in Section \ref{Section::OneNegCC} and Section \ref{Section::NegAndParaFaces}.

The first such statement considers subsets of the support containing all negative exponent vectors and those positive exponent vectors which lie on the separating hyperplane.

  \begin{prop}
\label{Prop_NonStrictSepHyp}
Let $f\colon \mathbb{R}^{n}_{>0} \to \mathbb{R}, \, x \mapsto \sum_{\mu \in \sigma(f)} c_{\mu}x^{\mu}$  be a signomial whose support has a separating hyperplane $\mathcal{H}_{v,a}$. For any subset $R \subseteq \sigma(f)$ such that $\sigma_-(f) \subseteq R$ and $\mathcal{H}_{v,a}  \cap \sigma_+(f) \subseteq R$, we have:
\begin{itemize}
\item[(i)] For all $U \in \mathcal{B}_0^-(f_{|R})$, there exists a unique $V \in  \mathcal{B}_0^-(f)$ such that $U \cap f^{-1}(\mathbb{R}_{<0}) \subseteq V$.
\item[(ii)] The map
\begin{align*}
\phi \colon  \mathcal{B}_0^-(f_{|R})  \to  \mathcal{B}_0^-(f), \quad U \mapsto 
      \begin{array}{l@{}}
        \text{connected component of $f^{-1}(\mathbb{R}_{<0})$} \\
        \text{that contains } U \cap f^{-1}(\mathbb{R}_{<0}) 
      \end{array}
\end{align*}
is well defined and bijective. In particular, it holds
\begin{align*}
 b_0 ( f_{|R}^{-1}(\mathbb{R}_{<0}) ) = b_0 ( f^{-1}(\mathbb{R}_{<0}) ).
\end{align*}
\end{itemize}
\end{prop}

Before giving the proof of Proposition \ref{Prop_NonStrictSepHyp}, we illustrate the idea behind it by the following example.

\begin{ex}
\label{Ex::Proof24}
The hyperplane $\mathcal{H}_{v,1}$ with $v=(1,0)$ is a separating hyperplane of the support~of
\[ f = 1+ x_1 + x_2 - 4x_1 x_2+ x_1  x_2^2.\]
The set $R = \{(1,1), (1,0),(0,1),(1,2)\}$ satisfies the assumptions in Proposition  \ref{Prop_NonStrictSepHyp}. If $x,y \in f^{-1}(\mathbb{R}_{<0})$ lie in the same connected component of $f_{|R}^{-1}(\mathbb{R}_{<0})$, then there exists a continuous path $\gamma\colon [0,1] \to \mathbb{R}^n_{>0}$ with $\gamma(0) = x, \, \gamma(1) = y$ and the paths 
\[\gamma_{v,x} \colon [0,1] \to \mathbb{R}^n_{>0}, \, t \mapsto t^v \ast x, \qquad \gamma_{v,y} \colon [0,1] \to \mathbb{R}^n_{>0}, \, t \mapsto t^v \ast y\]
are contained in $f^{-1}(\mathbb{R}_{<0})$. Using $v$, we push the path $\gamma$ such that its image becomes a subset of $ f^{-1}(\mathbb{R}_{<0})$. More precisely, we consider the path 
\[\tilde{\gamma}\colon [0,1] \to \mathbb{R}^n_{>0}, \quad s \mapsto t_0^v\ast \gamma(s)\]
for a fixed large enough $t_0 > 0$, see Figure \ref{FIG2_new}(b). Concatenating the paths $\tilde{\gamma}, \, \gamma_{v,x} \, \gamma_{v,y}$, we obtain a path between $x$ and $y$ in $f^{-1}(\mathbb{R}_{<0})$. Thus, $x$ and $y$ are contained in the same connected component of $f^{-1}(\mathbb{R}_{<0})$. We will use this observation to prove part (i) of Proposition \ref{Prop_NonStrictSepHyp}.

For $x_1,x_2 \in f^{-1}_{|R}(\mathbb{R}_{<0})$, if $t_0^{v} \ast x_1$ and $t_0^{v} \ast x_2$ are in the same connected component of $f^{-1}(\mathbb{R}_{<0})$ for some $t_0 >0$, then $x_1, x_2$ lie in the same connected component of  $f^{-1}_{|R}(\mathbb{R}_{<0})$, see Figure \ref{FIG2_new}(c). This fact will be used in the proof of Proposition \ref{Prop_NonStrictSepHyp} to show injectivity of the map~$\phi$.
\end{ex}

\begin{figure}[t]
\centering
\begin{minipage}[h]{0.32\textwidth}
\centering
\includegraphics[scale=0.5]{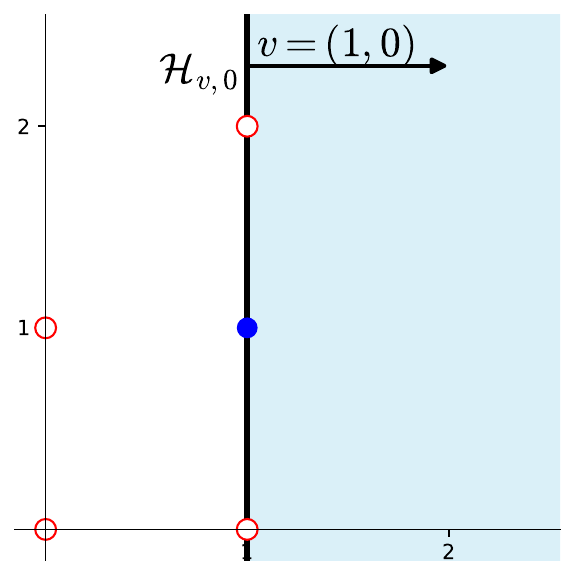}

{\small (a)}
\end{minipage}
\begin{minipage}[h]{0.32\textwidth}
\centering
\includegraphics[scale=0.5]{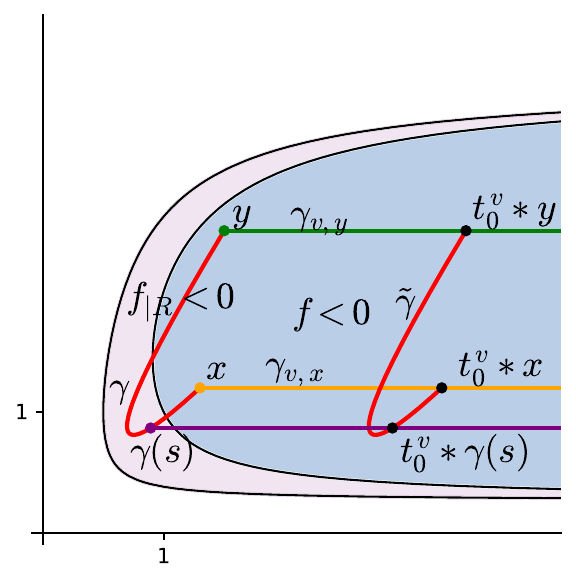}

{\small (b)}
\end{minipage}
\begin{minipage}[h]{0.32\textwidth}
\centering
\includegraphics[scale=0.5]{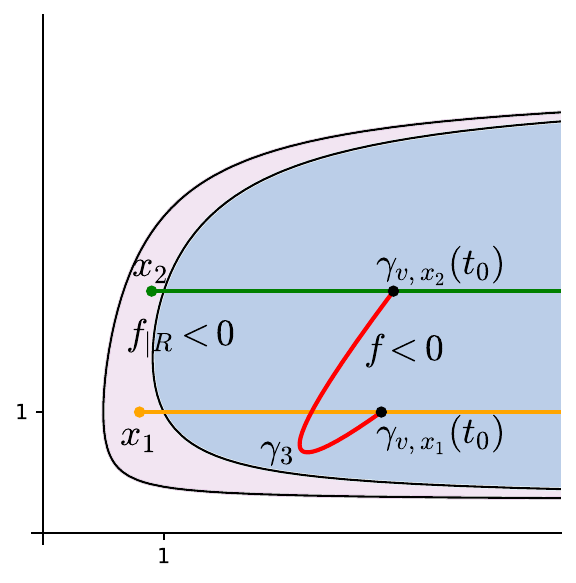}

{\small (c)}
\end{minipage}

\caption{{\small  (a) Exponent vectors of $f  = 1+ x_1 + x_2 - 4x_1 x_2+ x_1  x_2^2$ from Example~\ref{Ex::Proof24} and a separating hyperplane of $\sigma(f)$. Positive exponent vectors are depicted by red circles, the negative exponent vector by blue dot. (b),(c) An illustration of the paths from Example \ref{Ex::Proof24} and the proof of Proposition \ref{Prop_NonStrictSepHyp}.}}\label{FIG2_new}
\end{figure}

\begin{proof}[Proof of Proposition \ref{Prop_NonStrictSepHyp}]
First observe that
\begin{align}
\label{Eq_fFcontf}
f^{-1}(\mathbb{R}_{<0} ) \subseteq f_{|R}^{-1}(\mathbb{R}_{<0} ),
\end{align}
because $f_{|R}(x) \leq f(x)$ for all $x \in \mathbb{R}^{n}_{>0}$.

If $\sigma_+(f)  \subseteq R$, then $R = \sigma(f)$ and we are done. So we assume that there exists $\alpha_0 \in \sigma_+(f)$ such that $\alpha_0 \notin \mathcal{H}_{v,a}$. Since $\mathcal{H}_{v,a}$ is a separating hyperplane of $\sigma(f)$, we have
\begin{align}
\label{Eq::ProofProp_NonStrictSepHyp}
  v \cdot \beta  \geq  a \geq v \cdot \alpha \qquad \text{for all }  \beta \in \sigma_-(f), \,\alpha \in \sigma_+(f) \qquad \text{and} \quad a > v \cdot \alpha_0.
 \end{align}

For $x \in \mathbb{R}^{n}_{>0}$ , consider the path
\[ \gamma_{v,x}\colon [1,\infty) \to \mathbb{R}^n_{>0}, \quad t \mapsto t^v \ast x \]
and the univariate signomial  $f(t^v \ast x)$ as in $\eqref{Eq_InducedSignomial}$. Note that $f(\gamma_{v,x}(t)) =  f(t^v \ast x)$ for all $t \geq 1$.

If $v$ is a strict separating vector of $\sigma(f)$, then $f(t^v \ast x)$ has a negative leading coefficient. If $v$ is not strict, then 
\[ \LC( f(t^v\ast x)) = f_{|\mathcal{H}_{v,a}}(x) \quad \text{and} \quad f_{|R}(x) =  f_{|\mathcal{H}_{v,a}}(x) + \sum_{\alpha \in \sigma_+(f) \cap R \cap \mathcal{H}^{-,\circ}_{v,a}} c_\alpha x^\alpha. \]
Thus, if $f_{|R}(x) <0$, then  $ \LC( f(t^v\ast x)) < 0$.

By \eqref{Eq::ProofProp_NonStrictSepHyp},  $f(t^v\ast x)$ and $f_{|R}(t^v\ast x)$ have at most one sign change in their coefficient sign sequence. By Lemma \ref{Lemma_UniSignchanges}(i), we have:
\begin{align}
\label{Eq_ProofNegFace}
&\im \gamma_{v,x} \subseteq f^{-1}(\mathbb{R}_{<0}), & &\text{for all } x \in f^{-1}(\mathbb{R}_{<0}), \\
\label{Eq_ProofNegFace_ftilde}
 &\im \gamma_{v,x} \subseteq f_{|R}^{-1}(\mathbb{R}_{<0}) \quad \text{and} \quad  \gamma_{v,x}(t) \in f^{-1}(\mathbb{R}_{<0}), &  &\text{for all }  x \in f_{|R}^{-1}(\mathbb{R}_{<0}) < 0, \, t\gg 1.
\end{align}
This gives that each connected component of $f_{|R}^{-1}(\mathbb{R}_{<0})$ has a non-empty intersection with $f^{-1}(\mathbb{R}_{<0})$. For an illustration of the path $\gamma_{v,x}$ we refer to Figure \ref{FIG2_new}(b)(c).

To prove (i), let $U \in \mathcal{B}_0^-(f_{|R})$, $x,y \in U \cap f^{-1}(\mathbb{R}_{<0})$ and consider a continuous path 
\[ \gamma\colon [0,1] \to \mathbb{R}^{n}_{>0}\]
such that $\gamma(0) = x, \, \gamma(1) = y$ and $\gamma(s) \subseteq f_{|R}^{-1}(\mathbb{R}_{<0} )$ for all $s \in [0,1]$.

We construct now a path between $x$ and $y$ that is contained in $f^{-1}(\mathbb{R}_{<0})$. For a fixed $s \in [0,1]$, since $f_{|R}(\gamma(s)) <0$, and the signomial $f(t^v\ast \gamma(s))$ has exactly one sign change in its coefficient sign sequence, Descartes' rule of signs implies that $f(t^v\ast \gamma(s))$ has exactly one positive real root $\tau(s)$, which is simple. Now, the Implicit Function Theorem \cite{EDWARDS197356} implies that the function
\[\tau\colon [0,1] \to \mathbb{R}_{>0}, \quad s \mapsto \tau(s)\] 
is continuous. So $T := \max_{s \in [0,1]}\tau(s)$ exists and we have:
\[f(t_0^v\ast \gamma(s)) < 0 \qquad \text{for some } t_0 > \max\{1,T\} \text{ and all } s \in [0,1].\]
Thus, the path
\[  [0,1] \to \mathbb{R}^n_{>0}, \quad s \to t_0^v \ast \gamma(s) \]
connects $t_0^v \ast x$ and $t_0^v \ast y$, and is contained in  $f^{-1}(\mathbb{R}_{<0})$ (cf. Figure \ref{FIG2_new}(b)). By \eqref{Eq_ProofNegFace}, the paths $\gamma_{v,x}, \gamma_{v,y}$ are contained in $f^{-1}(\mathbb{R}_{<0})$ and join $x$ and $t_0 \ast x$, resp. $y$ and $t_0 \ast y$. Thus, if $x,y$ are in the same connected component of $f_{|R}^{-1}(\mathbb{R}_{<0})$, then they are also in the same connected component of $f^{-1}(\mathbb{R}_{<0})$, which gives (i).

By (i), the map $\phi$ is well defined. Now we show that $\phi$ is bijective. From \eqref{Eq_fFcontf}, it follows that every $W \in \mathcal{B}_0^{-}(f)$ lies in a connected component of  $f_{|R}^{-1}(\mathbb{R}_{<0} )$, which is a preimage of $W$ under $\phi$. Thus, $\phi$ is surjective.

To show injectivity of $\phi$, consider $U_1, U_2 \in \mathcal{B}_{0}^{-}(f_{|R})$ such that $\phi(U_1) = \phi(U_2)$ and let $x_1 \in U_1, x_2 \in U_2$. By \eqref{Eq_ProofNegFace_ftilde}
\begin{align*}
 \im \gamma_{v,x_1}  \subseteq f_{|R}^{-1}(\mathbb{R}_{<0}),  \quad  \im \gamma_{v,x_2} \subseteq f_{|R}^{-1}(\mathbb{R}_{<0}) \quad \text{and} \\
 \gamma_{v,x_1}(t_0) \in f^{-1}(\mathbb{R}_{<0}), \quad  \gamma_{v,x_2}(t_0) \in f^{-1}(\mathbb{R}_{<0}) \text{ for some } t_0\gg 1.
\end{align*}
Since $U_1 \cap f^{-1}(\mathbb{R}_{<0})$ and $U_2 \cap f^{-1}(\mathbb{R}_{<0})$ are in the same connected component of $f^{-1}(\mathbb{R}_{<0})$, there exists a continuous path $\gamma_3$ between $ \gamma_{v,x_1}(t_0)$ and $ \gamma_{v,x_2}(t_0)$ such that $\im \gamma_3 \subseteq f^{-1}(\mathbb{R}_{<0})$ (cf. Figure \ref{FIG2_new}(c)). By \eqref{Eq_fFcontf}, we have $\im \gamma_3  \subseteq f_{|R}^{-1}(\mathbb{R}_{<0})$. Thus,  $\gamma_{v,x_1}, \gamma_3$ and  $\gamma_{v,x_2}$ give a continuous path between $x_1$ and $x_2$ contained in $f_{|R}^{-1}(\mathbb{R}_{<0})$. Thus, $x,y$ lie in the same connected component of $f_{|R}^{-1}(\mathbb{R}_{<0})$, which implies that $U_1 = U_2$.
\end{proof}

In the following, we generalize the proof of Theorem \ref{Thm::FirstDescartes}(ii) to the case where the enclosing hyperplanes are not necessarily strict. To ease the reading of the proofs, we discuss first the notation we are going to use. Given a pair of enclosing hyperplanes $(\mathcal{H}_{v,a}, \mathcal{H}_{v,b})$ of $\sigma_+(f)$, we define
\begin{align}
\label{Eq::DefAB}
A = (\mathcal{H}^{+}_{v,a} \cap \sigma_-(f)) \cup \sigma_+(f), \qquad B = (\mathcal{H}^{-}_{v,b} \cap \sigma_-(f)) \cup \sigma_+(f). 
\end{align}
That is, both $A$ and $B$ contain all the positive exponent vectors of $f$. The set $A$ (resp. $B$) contains the negative exponent vectors from one (resp. the other) side of the area enclosed by the hyperplanes $\mathcal{H}_{v,a}, \mathcal{H}_{v,b}$. Note that by definition, $\mathcal{H}_{v,a}$ (resp $\mathcal{H}_{-v,-b}$) is a separating hyperplane of the support of the restricted signomials $f_{|A}$ (resp $f_{|B})$. If $(\mathcal{H}_{v,a}, \mathcal{H}_{v,b})$ is a pair of strict enclosing hyperplanes, then  $\mathcal{H}_{v,a}$ and $\mathcal{H}_{-v,-b}$ are  strict separating hyperplanes of $\sigma(f_{|A})$ and $\sigma(f_{|B})$ respectively. In this case, by Theorem \ref{Thm::FirstDescartes}(i) we have:
\begin{align}
\label{Eq:StrictEnclAB}
b_0\big( f_{|A}^{-1}(\mathbb{R}_{<0})\big) = 1, \qquad b_0\big( f_{|B}^{-1}(\mathbb{R}_{<0})\big) = 1.
\end{align}

\begin{ex}
\label{Ex:NewEx2}
For the signomial 
\[ f = 2 - 2 x_1^2 - 2 \, x_2^{2}  + 5 \, x_1 x_2  + x_1^{2} x_2^{2} -2 \, x_1^{3} x_2,\]
the pair of hyperplanes $(\mathcal{H}_{v,2},\mathcal{H}_{v,0})$ with $v=(1,0)$ is a pair of enclosing hyperplanes of $\sigma(f)$, see Figure~\ref{FIG3_new}(a). The sets $A,B$ from \eqref{Eq::DefAB} are given by
\begin{align*}
A &= (\mathcal{H}^{+}_{v,2} \cap \sigma_-(f)) \cup \sigma_+(f) =  \{ (2,0),(3,1),(2,2),(1,1),(0,0) \} \\ 
B &= (\mathcal{H}^{-}_{v,0} \cap \sigma_-(f)) \cup \sigma_+(f) =  \{ (0,2),(2,2),(1,1),(0,0) \}.
\end{align*}
 For fixed $x \in \mathbb{R}^{2}_{>0}$, the induced univariate signomials as in \eqref{Eq_InducedSignomial} equal
\begin{equation}
\label{Eq:NewExInduced}
\begin{aligned}
 f(t^1 x, t^0 y) &= -2 x_1^3x_2 t^3 + (x_1^2x_2^2-2x_1^2)t^2+5x_1x_2 t +(2-2x_2^2)\\
f_{|A}(t^1 x, t^0 y) &=  -2 x_1^3x_2 t^3 + (x_1^2x_2^2-2x_1^2)t^2+5x_1x_2 t +2\\
f_{|B}(t^{-1} x, t^0 y) &=    (2-2x_2^2) +5x_1x_2 t^{-1} +  x_1^2x_2^2 t^{-2} .
 \end{aligned}
 \end{equation}
 By construction, the signomials $f_{|A}(t^1 x, t^0 y),\, f_{|B}(t^{-1} x, t^0 y)$ have at most one sign change in their coefficient sequence. Using this, we will show in the proof of  Proposition \ref{Prop_StripBound0} that for any $x \in f^{-1}(\mathbb{R}_{<0})$ one of the paths
 \begin{align}
 \label{Eq::NewExPath}
  \gamma_v \colon[1,\infty) \to \mathbb{R}^n_{>0} , \quad t \mapsto t^v \ast x, \qquad \gamma_{-v} \colon[1,\infty) \to \mathbb{R}^n_{>0} , \quad t \mapsto t^{-v} \ast x
 \end{align}
 connects x to a point in $f^{-1}_{|A}(\mathbb{R}_{<0}) \cup f^{-1}_{|B}(\mathbb{R}_{<0})$, see Figure~\ref{FIG3_new}(b). One technical difficulty we must address in the proof of Proposition  \ref{Prop_StripBound0} is that the leading and trailing coefficients of induced signomials in \eqref{Eq:NewExInduced} depend on $x$. For example,
 \begin{align*}
  &   \SC(f(t^{1} x, t^0 y)) = \LC(f_{|B}(t^{-1} x, t^0 y))  = (2-2x_2^2) < 0,   &\SC(f_{|A}(t^{1} x, t^0 y)) = 2  \quad  &\text{ if } 2-2x_2^2 <  0, \\
 &   \SC(f(t^{1} x, t^0 y)) = \LC(f_{|B}(t^{-1} x, t^0 y))  = 5 x_1 x_2 > 0,   &\SC(f_{|A}(t^{1} x, t^0 y)) = 2  \quad &\text{ if } 2-2x_2^2 = 0. \\
 \end{align*}
\end{ex}

 \begin{figure}[t]
\centering
\begin{minipage}[h]{0.4\textwidth}
\centering
\includegraphics[scale=0.65]{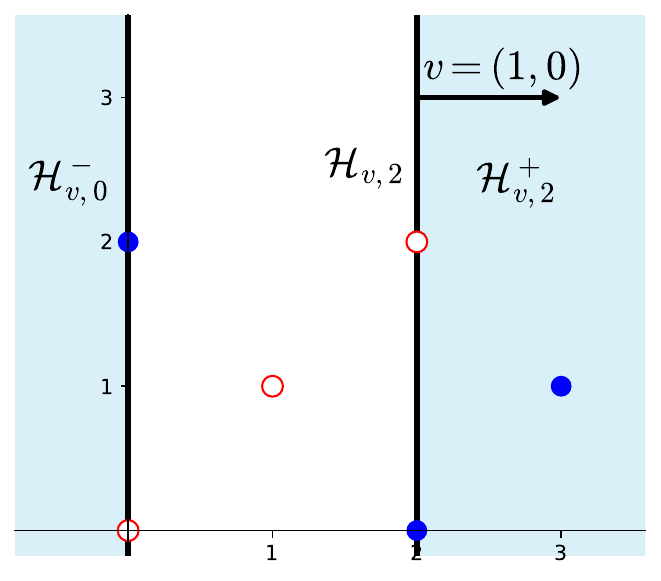}

{\small (a)}
\end{minipage}
\hspace{25pt}
\begin{minipage}[h]{0.4\textwidth}
\centering
\includegraphics[scale=0.65]{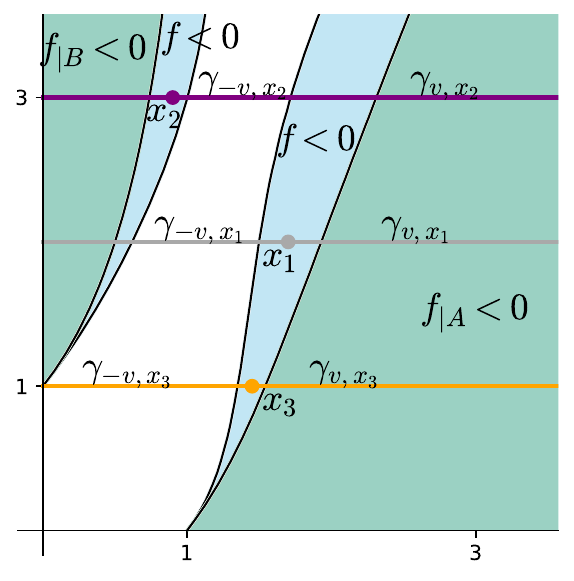}

{\small (b)}
\end{minipage}

\caption{{\small  (a) A pair of enclosing hyperplanes of the support of the signomial from Example \ref{Ex:NewEx2}. (b) Paths as defined in \eqref{Eq::NewExPath} connecting $x_1,x_2,x_3$ to a point in $f_{|A}^{-1}(\mathbb{R}_{<0}) \cap f_{|B}^{-1}(\mathbb{R}_{<0})$.   }}\label{FIG3_new}
\end{figure}
 
 In the next proposition, we consider subsets $R \subseteq A$, $S \subseteq B$ such that $f_{|A}$  and $R$ (resp. $f_{|B}$ and $S$) satisfy the hypothesis of Proposition \ref{Prop_NonStrictSepHyp}.
 
\begin{prop}
\label{Prop_StripBound0}
Let $f\colon \mathbb{R}^{n}_{>0} \to \mathbb{R}, \, x \mapsto \sum_{\mu \in \sigma(f)} c_{\mu}x^{\mu}$  be a signomial and $(\mathcal{H}_{v,a}, \mathcal{H}_{v,b})$ be a pair of enclosing hyperplanes of $\sigma_+(f)$. Let $A,B \subseteq \sigma(f)$ be as in \eqref{Eq::DefAB} and let $R \subseteq A, \, S \subseteq B$ such that 
\begin{itemize}
\item $R \cap \sigma_-(f) = A \cap \sigma_-(f), \qquad  \mathcal{H}_{v,a} \cap \sigma_+(f) \subseteq R$,
\item  $S \cap \sigma_-(f) = B \cap \sigma_-(f), \qquad  \mathcal{H}_{v,b} \cap \sigma_+(f) \subseteq S$.
\end{itemize}
 Then the map
\begin{align*}
\varphi \colon &\mathcal{B}_0^-(f_{|R}) \sqcup \mathcal{B}_0^-(f_{|S}) \to \mathcal{B}_0^-(f), \\
 U &\mapsto \begin{cases} 
 \text{ connected component of $f^{-1}( \mathbb{R}_{<0})$ that contains } U \cap f_{|A}^{-1}(\mathbb{R}_{<0}),  & \textrm{if } U \in \mathcal{B}_0^-(f_{|R}) , \\
\text{ connected component of $f^{-1}( \mathbb{R}_{<0})$ that contains }U \cap f_{|B}^{-1}(\mathbb{R}_{<0}), & \textrm{if } U \in \mathcal{B}_0^-(f_{|S }). \end{cases}
\end{align*}
is well defined and surjective. In particular, we have
\[ b_0\big( f^{-1}(\mathbb{R}_{<0} )\big) \leq  b_0\big( f_{|R }^{-1}(\mathbb{R}_{<0}) \big) +b_0\big( f_{|S }^{-1}(\mathbb{R}_{<0}) \big),\]
\end{prop}

\begin{proof}
The idea of the proof follows closely the arguments in \cite[Theorem 3.8]{DescartesHypPlane}. Since $f_{|A}$ and $f_{|B}$ are obtained by removing some of the monomials with negative coefficient from $f$, $f(x) \leq f_{|A}(x)$ and $f(x) \leq f_{|B}(x)$ for all $x \in \mathbb{R}^{n}_{>0}$. This implies:
\begin{align*}
f_{|A}^{-1}(\mathbb{R}_{<0}) \subseteq f^{-1}(\mathbb{R}_{<0}), \qquad \text{and} \qquad f_{|B}^{-1}(\mathbb{R}_{<0}) \subseteq f^{-1}(\mathbb{R}_{<0}).
\end{align*}
Thus, every connected component of $f_{|A}^{-1}(\mathbb{R}_{<0})$ and $f_{|B}^{-1}(\mathbb{R}_{<0})$ is contained in a unique connected component of $f^{-1}(\mathbb{R}_{<0}) $. By this observation, the map
\[ \psi\colon \mathcal{B}_0^-(f_{|A}) \sqcup \mathcal{B}_0^-(f_{|B}) \to \mathcal{B}_0^-(f), \quad V \mapsto \text{connected component of $f^{-1}(\mathbb{R}_{<0})$ that contains } V\]
is well defined.

From Proposition \ref{Prop_NonStrictSepHyp} applied to $f_{|A}$ and $R$ or $f_{|B}$ and $S$, it follows that the map
\begin{align*}
\phi \colon &\mathcal{B}_0^-(f_{|R}) \sqcup \mathcal{B}_0^-(f_{|S}) \to \mathcal{B}_0^-(f_{|A}) \sqcup \mathcal{B}_0^-(f_{|B}) , \\
 U &\mapsto \begin{cases} 
 \text{ connected component of $f_{|A}^{-1}( \mathbb{R}_{<0})$ that contains } U \cap f_{|A}^{-1}(\mathbb{R}_{<0}),  & \textrm{if } U \in \mathcal{B}_0^-(f_{|R}) , \\
\text{ connected component of $f_{|B}^{-1}( \mathbb{R}_{<0})$ that contains }U \cap f_{|B}^{-1}(\mathbb{R}_{<0}), & \textrm{if } U \in \mathcal{B}_0^-(f_{|S }). \end{cases}
\end{align*}
is well defined and bijective. Since $\varphi$ is the composition of $\phi$ and $\psi$, it is enough to show that $\psi$ is surjective.  

To prove that $\psi$ is surjective, for $W \in \mathcal{B}^-_0(f)$ and any $x \in W$, we show that one of the two paths
\begin{align*}
\gamma_v \colon [1,\infty) \to \mathbb{R}^n_{>0}, \quad t \mapsto t^v \ast x, \qquad \gamma_{-v} \colon [1,\infty) \to \mathbb{R}^n_{>0}, \quad t \mapsto t^{-v} \ast x
\end{align*}
\begin{itemize}
\item[(a)] connects $x$ to a point $y \in f_{|A}^{-1}(\mathbb{R}_{<0}) \cup f_{|B}^{-1}(\mathbb{R}_{<0})$, and
\item[(b)] the image of the path is contained in $ f^{-1}(\mathbb{R}_{<0})$.
\end{itemize}
For an example of these paths, we refer to Figure \ref{FIG3_new}. If (a) and (b) are satisfied, then any connected component $V \in \mathcal{B}_0^-(f_{|A}) \sqcup \mathcal{B}_0^-(f_{|B}) $ that contains $y$ will be a preimage of $W$ under $\psi$. To see this, we define
 \begin{align*}
 S_{x,a}(t) := \sum_{\mu\in \sigma(f), \, a \leq v \cdot \mu} c_\mu x^\mu t^{v\cdot \mu}, \qquad  S_{x,b}(t) :=   \begin{cases}
 \sum\limits_{\mu\in \sigma(f), \,  v \cdot \mu \leq b} c_\mu x^\mu t^{v \cdot \mu}, \text{ if } a \neq b\\
 \sum\limits_{\mu\in \sigma(f), \,  v \cdot \mu < b} c_\mu x^\mu t^{v \cdot \mu}, \text{ if } a = b
\end{cases}.
\end{align*}
 and  consider the signomials
\begin{align*}
&\tilde{f}_A(t) := f_{|A}(t^{v} \ast x) = S_{x,a}(t) +  \sum_{\mu\in \sigma_+(f),  v \cdot \mu < a} c_\mu x^\mu t^{v\cdot \mu},  \\
&\tilde{f}_B(t) := f_{|B}(t^{v} \ast x) = \sum_{\mu\in \sigma_+(f),  b < v \cdot \mu } c_\mu x^\mu t^{v \cdot \mu} + S_{x,b}(t).
\end{align*}
A simple argument shows that
\begin{equation}
\label{Eq_ProofA}
\begin{aligned}
 \text{If } \LC(\tilde{f}_A) < 0, &\text{ then } \tilde{f}_A(t_A) < 0 \text{ for some } t_A > 1  \text{ and } \gamma_v \text{ connects } x \text{ to } f_{|A}^{-1}(\mathbb{R}_{<0}). \\
  \text{If } \SC(\tilde{f}_B) < 0, &\text{ then }\tilde{f}_B(t_B) < 0 \text{ for some } t_B < 1 \text{ and } \gamma_{-v} \text{ connects } x \text{ to } f_{|B}^{-1}(\mathbb{R}_{<0}).
\end{aligned}
\end{equation}

Since $(\mathcal{H}_{v,a},\mathcal{H}_{v,b})$ is a pair of enclosing hyperplanes of $\sigma_+(f)$, the univariate signomial
\begin{align*}
&\tilde{f}(t) := f(t^{v} \ast x) = S_{x,a}(t) +  \sum_{\substack{\mu \in \sigma_+(f), \\ b < v \cdot \mu < a}} c_\mu x^\mu t^{v \cdot \mu} + S_{x,b}(t) 
\end{align*}
has at most two sign changes in its coefficient sign sequence. We denote the number of sign changes by $\signvar(\tilde{f})$. By Lemma \ref{Lemma_UniSignchanges}, we have:
\begin{align} \nonumber
&\im(\gamma_{v}) \subseteq f^{-1}(\mathbb{R}_{<0})\, \text{or} \, \im(\gamma_{-v}) \subseteq f^{-1}(\mathbb{R}_{<0}), & &\text{ if } \LC(\tilde{f}) < 0 \text{ and }  \SC(\tilde{f}) < 0,\\ 
\label{Eq::Proof_Prop_StripBound0}
&\im(\gamma_{v}) \subseteq f^{-1}(\mathbb{R}_{<0}), & &\text{ if } \LC(\tilde{f}) < 0 \text{ and } \signvar(\tilde{f}) \leq 1, \\
&\im(\gamma_{-v}) \subseteq f^{-1}(\mathbb{R}_{<0}), & &\text{ if } \SC(\tilde{f}) < 0 \text{ and } \signvar(\tilde{f}) \leq 1.  \nonumber
\end{align}

In view of \eqref{Eq_ProofA} and \eqref{Eq::Proof_Prop_StripBound0}, to obtain (a) and (b) all we need is to show that one of the following holds:
\begin{itemize}
\item[(I)]  $\LC(\tilde{f})=\LC(\tilde{f}_A)  < 0 \text{ and }  \SC(\tilde{f})=\SC(\tilde{f}_B)  < 0$,
\item[(II)] $\LC(\tilde{f})=\LC(\tilde{f}_A) < 0 \text{ and } \signvar(\tilde{f}) \leq 1$,
\item[(III)] $\SC(\tilde{f})=\SC(\tilde{f}_B) < 0 \text{ and } \signvar(\tilde{f}) \leq 1$. 
\end{itemize}

The leading and trailing coefficients $\LC(\tilde{f}),\LC(\tilde{f}_A),\SC(\tilde{f})$ and  $\SC(\tilde{f}_B)$ depend on the signs of $S_{x,a}(t), \, S_{x,b}(t)$. If $S_{x,a}(t)$ (resp. $S_{x,b}(t)$) is not the zero polynomial, then $\LC(\tilde{f}) = \LC(\tilde{f}_A) = \LC( S_{x,a})$ (resp. $\SC(\tilde{f}) = \SC(\tilde{f}_B) = \SC(S_{x,b})$). As $\tilde{f}(1) <0$, we have that at least one of $S_{x,a}(t)$ and $S_{x,b}(t)$  is not the zero polynomial, furthermore $\LC(S_{x,a}) < 0$ or $\SC(S_{x,b})< 0$. We have the following cases:

\begin{itemize}
\item If $\LC(S_{x,a}) < 0$ and $\SC(S_{x,b}) < 0$, then $S_{x,a} \not\equiv 0$, $S_{x,b} \not\equiv 0$  and (I) is satisfied. 
\item If $\LC(S_{x,a}) < 0$ and $\SC(S_{x,b}) > 0$ or $S_{x,b} \equiv 0$, then (II) holds.
\item  If $\SC(S_{x,b}) < 0$ and $\LC(S_{x,a}) > 0$ or $S_{x,a} \equiv 0$, then (III) holds.
\end{itemize}
\end{proof}

If $(\mathcal{H}_{v,a}, \mathcal{H}_{v,b})$ is a pair of enclosing hyperplanes of $\sigma_+(f)$, then by Proposition \ref{Prop_NonStrictSepHyp}(iii), the number of negative connected components of $f_{|R}$ and $f_{|S}$ does not depend on the choice of $R$ and $S$, as long as they satisfy the assumptions of Proposition \ref{Prop_StripBound0}. So to reduce the computational cost of finding $b_0(f_{|R}^{-1}(\mathbb{R}_{<0}))$ and $b_0(f_{|S}^{-1}(\mathbb{R}_{<0}))$ one might choose the sets $R, S$ as small as possible, that is $R = \mathcal{H}_{v,a}^+ \cap \sigma(f)$ and $S = \mathcal{H}_{v,b}^- \cap \sigma(f)$. However, if $R$ and $S$ are as large as possible, that is $R = A$ and $S = B$, then one can refine the bound from Proposition \ref{Prop_StripBound0} by identifying negative connected components of $f_{|A}$ and $f_{|B}$ that intersect. To make this precise, consider the bipartite graph with vertex set and edges defined as
\begin{equation} \label{Eq::GraphABVertex}
\begin{aligned}
\mathcal{B}_{A,B} &:= \mathcal{B}_0^-(f_{|A}) \sqcup \mathcal{B}_0^-(f_{|B}), \\ 
\mathcal{E}_{A,B} &:= \{ (U,V) \mid U \in \mathcal{B}_0^-(f_{|A} ), V \in \mathcal{B}_0^-(f_{|B}) \colon U \cap V \neq \emptyset \}.
\end{aligned}
\end{equation}

\begin{prop}
\label{Lema_StripBound}
Let $f\colon \mathbb{R}^{n}_{>0} \to \mathbb{R}, \, x \mapsto \sum_{\mu \in \sigma(f)} c_{\mu}x^{\mu}$  be a signomial, $(\mathcal{H}_{v,a}, \mathcal{H}_{v,b})$ a pair of enclosing hyperplanes  of $\sigma_+(f)$, and let $A,B \subseteq \sigma(f)$ be as in \eqref{Eq::DefAB}. Then 
\[ b_0\big( f^{-1}(\mathbb{R}_{<0} )\big) \leq C \leq b_0\big( f_{|A }^{-1}(\mathbb{R}_{<0}) \big) +b_0\big( f_{|B }^{-1}(\mathbb{R}_{<0}) \big),\]
where $C$ denotes the number of connected components of the graph $(\mathcal{B}_{A,B}, \mathcal{E}_{A,B})$ from \eqref{Eq::GraphABVertex}.
\end{prop}

\begin{proof}
The inequality $C \leq b_0\big( f_{|A }^{-1}(\mathbb{R}_{<0}) \big) +b_0\big( f_{|B }^{-1}(\mathbb{R}_{<0})\big)$ follows from the fact that a graph cannot have more connected components than the number of its vertices.

Let $\varphi$ be the surjective map from Proposition \ref{Prop_StripBound0} with $R = A, S = B$. If the intersection of two connected sets is non-empty, then their union is connected. This gives that if $U,V \in \mathcal{B}_{A,B}$ lie in the same connected component of the graph $(\mathcal{B}_{A,B},\mathcal{E}_{A,B})$, then $U$ and $V$ are contained in the same connected component of $f^{-1}(\mathbb{R}_{<0})$. This yields a well defined map
\[\tilde{\varphi}\colon \mathcal{B}_{A,B}\big/{\sim} \to \mathcal{B}_0^-(f), \quad [U] \mapsto \varphi(U) = \text{ connected component of $f^{-1}(\mathbb{R}_{<0})$ that contains $U$}.\]
where $\sim$ denotes the equivalence relation that identifies two elements of $\mathcal{B}_{A,B}$ if they lie in the same connected component of the graph ($\mathcal{B}_{A,B},\mathcal{E}_{A,B})$, and $[U]$ denotes the equivalence class of $U$ under this relation. Since $\varphi$ is surjective, so is $\tilde{\varphi}$, which gives $b_0\big( f^{-1}(\mathbb{R}_{<0} )\big) \leq C$.
\end{proof}

\begin{remark}
\label{Remark::StrictEnclosing}
If $(\mathcal{H}_{v,a}, \mathcal{H}_{v,b})$ is a pair of strict enclosing hyperplanes of $\sigma_+(f)$, then from \eqref{Eq:StrictEnclAB} and Proposition \ref{Lema_StripBound} follows that 
\[ b_0\big( f^{-1}(\mathbb{R}_{<0} )\big) \leq 2,\]
which recovers the bound from Theorem \ref{Thm::FirstDescartes}(ii).
\end{remark}

\begin{ex}
\label{Ex_ParaFaces}
We revisit the signomial $f$ in \eqref{Eq_Example2} from Example \ref{Ex::ExRunning}, whose support and a pair of enclosing hyperplanes  $(\mathcal{H}_{v,2}, \mathcal{H}_{v,0})$, $v=(1,0)$ of $\sigma_+(f)$ are depicted in Figure \ref{FIG1}(a). The signomials $f_{|A}, \, f_{|B}$ are shown in \eqref{Eq_Example1} and \eqref{Eq_Example3}.

Using the \texttt{Maple} \cite{maple} function \texttt{IsEmpty()}, one can check that $f_{|A}^{-1}(\mathbb{R}_{<0}) \cap  f_{|B}^{-1}(\mathbb{R}_{<0}) = \emptyset$. Thus, the graph $( \mathcal{B}_{A,B}, \mathcal{E}_{A,B})$ does not have any edges and the bound $C$ on $b_0\big( f^{-1}(\mathbb{R}_{<0} )\big)$ provided by Proposition \ref{Lema_StripBound} equals
\[ C = b_0\big( f_{|A}^{-1}(\mathbb{R}_{<0} )\big) + b_0\big( f_{|B}^{-1}(\mathbb{R}_{<0} )\big). \]

Since $\sigma(f_{|A})$ has a strict separating hyperplane, $b_0\big( f_{|A}^{-1}(\mathbb{R}_{<0} )\big) = 1$ by Theorem \ref{Thm::FirstDescartes}(i). In Example \ref{Ex_NegFace}, we show that $b_0\big( f_{|B}^{-1}(\mathbb{R}_{<0} )\big) = 2$. Thus, Proposition \ref{Lema_StripBound} gives the bound $C = 3$. In this case, $C$ equals the number of negative connected components of $f$. In Example \ref{Ex_Running}, we will consider a signomial where the bound $C$ is not sharp.
   
Using the function \texttt{implicit\_plot()} from  \texttt{SageMath} \cite{sagemath}, we depicted the sets $f^{-1}(\mathbb{R}_{<0})$, $f_{|A}^{-1}(\mathbb{R}_{<0})$ and   $f_{|B}^{-1}(\mathbb{R}_{<0})$ in Figure \ref{FIG2}(a),(b).
\end{ex}

\begin{figure}[t]
\centering
\begin{minipage}[h]{0.3\textwidth}
\centering
\includegraphics[scale=0.5]{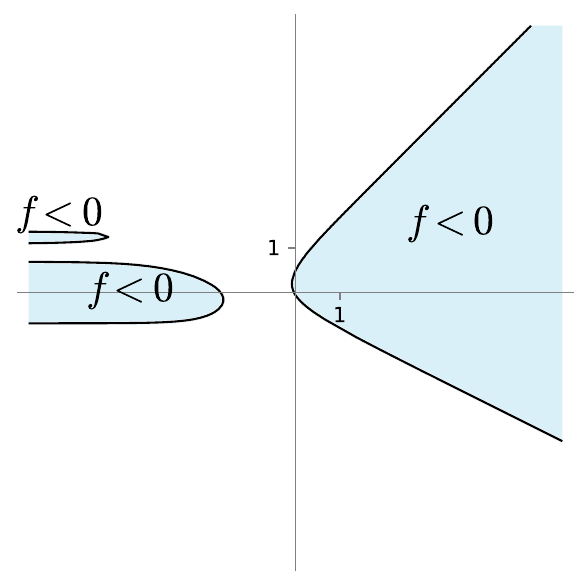}

{\small (a)}
\end{minipage}
\begin{minipage}[h]{0.3\textwidth}
\centering
\includegraphics[scale=0.5]{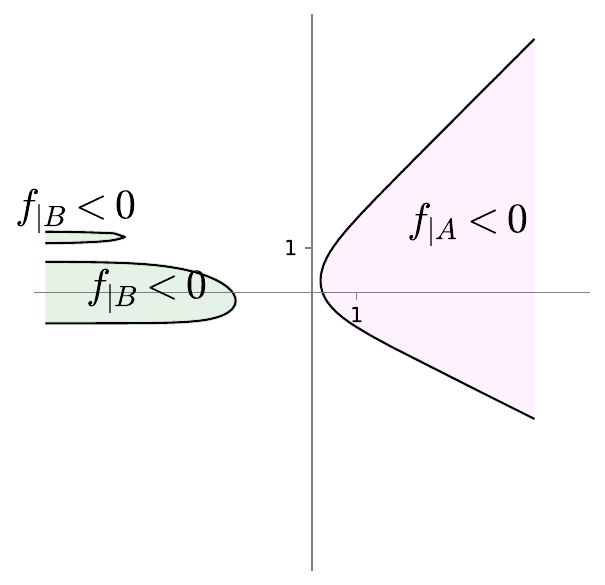}

{\small (b)}
\end{minipage}
\begin{minipage}[h]{0.3\textwidth}
\centering
\includegraphics[scale=0.5]{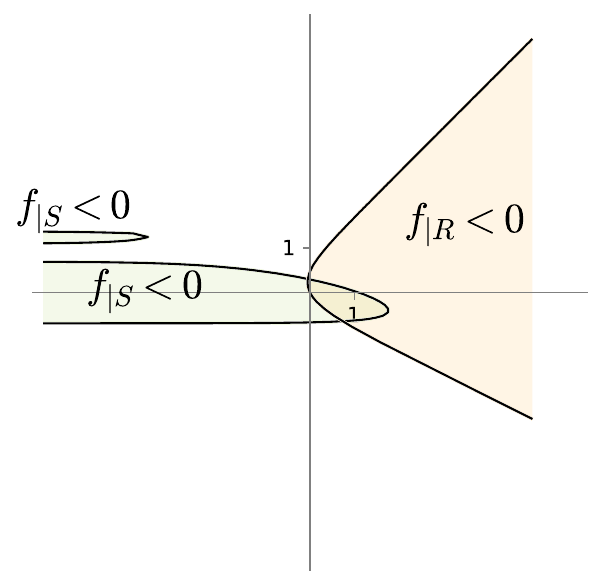}

{\small (c)}
\end{minipage}
\caption{{\small  Negative connected components of $f$,$f_{|A}$,$f_{|B}$,$f_{|R}$,$f_{|S}$ from Example \ref{Ex_ParaFaces} and Example \ref{Ex::RSCounterExample}}. For better visibility, the figures show the images of these sets under the coordinate-wise natural logarithm map $\mathbb{R}^2_{>0} \to \mathbb{R}^2, (x,y) \mapsto (\log(x),\log(y) )$. } \label{FIG2} 
\end{figure}

One might wonder if, for other choices of $R$ and $S$, the bound from Proposition \ref{Prop_StripBound0} could be refined using a similar idea as in the proof of Proposition \ref{Lema_StripBound}. We will show that such a refinement is possible in the special case where all the exponent vectors are contained in the enclosing hyperplanes, or in other words the exponent vectors lie on parallel faces of the Newton polytope (Proposition \ref{Thm_ParallelFaces}). Example \ref{Ex::RSCounterExample} shows that Proposition \ref{Lema_StripBound} might fail for other choices of $R$ and $S$.

\begin{ex}
\label{Ex::RSCounterExample}
Consider again the signomial $f$ in \eqref{Eq_Example2} from Example \ref{Ex::ExRunning} and choose
\[ R =  \{ (3,2),(2,1),(2,3),(1,3),(2,0) \}, \quad S =   \{ (0,3),(0,1),(1,3),(0,4),(0,2),(0,0)\}.\]
Similarly to \eqref{Eq::GraphABVertex}, we investigate the graph whose vertices are
\begin{align*}
\mathcal{B}_{R,S} &:= \mathcal{B}_0^-(f_{|R}) \sqcup \mathcal{B}_0^-(f_{|S})  \text{ and edges} \\ 
\mathcal{E}_{R,S} &:= \{ (U,V) \mid U \in \mathcal{B}_0^-(f_{|R} ), V \in \mathcal{B}_0^-(f_{|S}) \colon U \cap V \neq \emptyset \}.
\end{align*}
Since $\sigma(f_{|R})$ has a strict separating hyperplane, $f_{|R}$ has one negative connected component. From Theorem  \ref{Thm_NegFace}, it will follow that $f_{|S}$ has two negative connected components. Thus, the graph $(\mathcal{B}_{R,S},\mathcal{E}_{R,S})$ has three vertices. One can check that $(2,1) \in f_{|R}^{-1}(\mathbb{R}_{<0}) \cap  f_{|S}^{-1}(\mathbb{R}_{<0}) \neq \emptyset$. This implies that the graph has at most two connected components. Since $b_0( f^{-1}(\mathbb{R}_{<0} ) ) = 3$, the number of connected components of the graph $(\mathcal{B}_{R,S},\mathcal{E}_{R,S})$ cannot be an upper bound on $b_0( f^{-1}(\mathbb{R}_{<0} ))$. In Figure \ref{FIG2}(c), we depicted the sets $f_{|R}^{-1}(\mathbb{R}_{<0} ),f_{|S}^{-1}(\mathbb{R}_{<0} )$.
\end{ex}

To compute the bound $C$ in Proposition \ref{Lema_StripBound}, one should check whether the negative connected components of two signomials intersect. We finish this subsection with a criterion on the signed supports of two signomials guaranteeing that the intersection of their negative connected components is non-empty. 

\begin{prop}
\label{Prop_Box}
Let $f,g \colon \mathbb{R}^{n}_{>0} \to \mathbb{R}$  be signomials. If there exist negative exponent vectors $\beta_1 \in \sigma_-(f), \, \beta_2 \in \sigma_-(g)$ such that 
\[ \Conv\big( \{ \beta_1, \beta_2 \}\big) \cap \Conv\big( \sigma_+(f) \cup \sigma_+(g) \big) = \emptyset,\]
then  $f^{-1}(\mathbb{R}_{<0}) \cap g^{-1}(\mathbb{R}_{<0}) \neq \emptyset$.
\end{prop}

\begin{proof}
We start by observing
\begin{align}
\label{Eq_ProofBox1_Prop}
 f(x)  \leq f_{|\sigma_+(f) \cup \{ \beta_1 \} }(x), \quad \text{and} \quad g(x)  \leq g_{|\sigma_+(g) \cup \{ \beta_2 \}}(x) \quad \text{ for all } x \in \mathbb{R}^{n}_{>0},
\end{align}
since $f, f_{|\sigma_+(f) \cup \{ \beta_1 \} }$ (resp. $g,  g_{|\sigma_+(g) \cup \{ \beta_2 \}}$) differ only in monomials with negative coefficients. 

As non-intersecting closed convex sets can be separated by an affine hyperplane see e.g. \cite[Section 2.2, Theorem 1]{grunbaum2003convex}, from $\Conv( \{ \beta_1, \beta_2 \}) \cap \Conv( \sigma_+(f) \cup \sigma_+(g) ) = \emptyset$ it follows that there exists $w \in \mathbb{R}^{n}$ such that
\[ w \cdot \beta_1 > w \cdot \alpha, \quad w \cdot \beta_2 > w \cdot \alpha, \quad \text{for all } \alpha \in \sigma_+(f) \cup \sigma_+(g).\]
For a fixed $x \in \mathbb{R}^{n}_{>0}$, both univariate signomials
\[ f_{|\sigma_+(f) \cup \{ \beta_1 \}}(t^w\ast x), \qquad  g_{|\sigma_+(g) \cup \{ \beta_2 \}} (t^w\ast x)\]
have negative leading coefficients. Thus, there exists $t_0 \gg 0$ such that
\[ f_{|\sigma_+(f) \cup \{ \beta_1 \}}(t_0^w \ast x) < 0 \quad \text{ and } \quad  g_{|\sigma_+(g) \cup \{ \beta_2 \}}(t_0^w \ast x) < 0.\]
By \eqref{Eq_ProofBox1_Prop}, we have $t_0^w \ast x \in f_{|\sigma_+(f) \cup \{ \beta_1 \}}^{-1}(\mathbb{R}_{<0}) \cap g_{|\sigma_+(g) \cup \{ \beta_2 \}}^{-1}(\mathbb{R}_{<0}) \subseteq f^{-1}(\mathbb{R}_{<0}) \cap g^{-1}(\mathbb{R}_{<0})$, which completes the proof.
\end{proof}

\subsection{One negative connected component}
\label{Section::OneNegCC}
Building on the results of Section \ref{Section::NonStrict}, we describe conditions on the signs of the coefficients of $f$ and its support $\sigma(f)$ that guarantee that $f$ has one negative connected component. Similarly to Theorem \ref{Thm::FirstDescartes}(i), for polynomials satisfying these conditions the number of negative connected components does not depend on the values of the coefficients but only on their signs.

\begin{thm}
\label{Thm_Box}
Let $f\colon \mathbb{R}^{n}_{>0} \to \mathbb{R}, \, x \mapsto \sum_{\mu \in \sigma(f)} c_{\mu}x^{\mu}$  be a signomial such that $\sigma_+(f)$ has a pair of strict enclosing hyperplanes $(\mathcal{H}_{v,a}, \mathcal{H}_{v,b})$. Assume that there exist negative exponent vectors $\beta_1, \beta_2 \in \sigma_-(f)$ such that 
\begin{itemize}
\item $\beta_1 \in \mathcal{H}^{+}_{v,a}$, $\beta_2 \in \mathcal{H}^{-}_{v,b}$ and 
\item $\Conv( \{ \beta_1, \beta_2 \}) \cap \Conv( \sigma_+(f) ) = \emptyset$.
\end{itemize}
Then $f^{-1}(\mathbb{R}_{<0})$ is non-empty and connected.
\end{thm}

\begin{proof}
Let  $A, B \subseteq \sigma(f)$ as in \eqref{Eq::DefAB}. Since $(\mathcal{H}_{v,a}, \mathcal{H}_{v,b})$ is a pair of strict enclosing hyperplanes of $\sigma_+(f)$, by  \eqref{Eq:StrictEnclAB} we have:
\[b_0\big( f_{|A}^{-1}(\mathbb{R}_{<0})\big) = 1, \qquad b_0\big( f_{|B}^{-1}(\mathbb{R}_{<0})\big) = 1.\]
The assumptions of the theorem are equivalent to $\beta_1 \in \sigma_-(f_{|A}), \, \beta_2 \in \sigma_-(f_{|B})$ and $\Conv( \{ \beta_1, \beta_2 \}) \cap \Conv( \sigma_+(f_{|A}) \cup \sigma_+(f_{|B})) = \emptyset$. Thus, by Proposition \ref{Prop_Box}, $f^{-1}_{|A}(\mathbb{R}_{<0}) \cap f^{-1}_{|B}(\mathbb{R}_{<0}) \neq \emptyset$. Now, from Proposition \ref{Lema_StripBound} it follows that $b_0( f^{-1}(\mathbb{R}_{<0})) = 1$.
\end{proof}

\begin{ex}
\label{Ex_Box}
Consider the signomial $f =-x^{4} y^{4} + 10 \, x^{3} y^{3} - 10 \, x^{4} - 10 \, y^{4} + 7 \, x y + 5 \, x - 1$. Figure \ref{FIG3}(a) displays the exponent vectors of $f$. The pair $(\mathcal{H}_{v,3.5}, \, \mathcal{H}_{v,0.5})$, $v=(1,0)$ is a pair of strict enclosing hyperplanes of $\sigma_+(f)$, furthermore
\begin{align}
\label{Eq::Ex_Box}
 \Conv( \{ (0,4),(4,4) \} ) \cap \Conv(\sigma_+(f) ) = \emptyset.
 \end{align}
Thus, $f^{-1}(\mathbb{R}_{<0})$ is connected by Theorem \ref{Thm_Box}. Figure \ref{FIG3}(b),(c) show $f^{-1}(\mathbb{R}_{<0})$, $f_{|A}^{-1}(\mathbb{R}_{<0})$ and $f_{|B}^{-1}(\mathbb{R}_{<0})$, where $f_{|A} =-x^{4} y^{4} + 10 \, x^{3} y^{3} - 10 \, x^{4}  + 7 \, x y + 5 \, x  $ and $f_{|B}=10 \, x^{3} y^{3}  - 10 \, y^{4} + 7 \, x y + 5 \, x - 1$.
\end{ex}

\begin{figure}[t]
\centering
\begin{minipage}[h]{0.3\textwidth}
\centering
\includegraphics[scale=0.5]{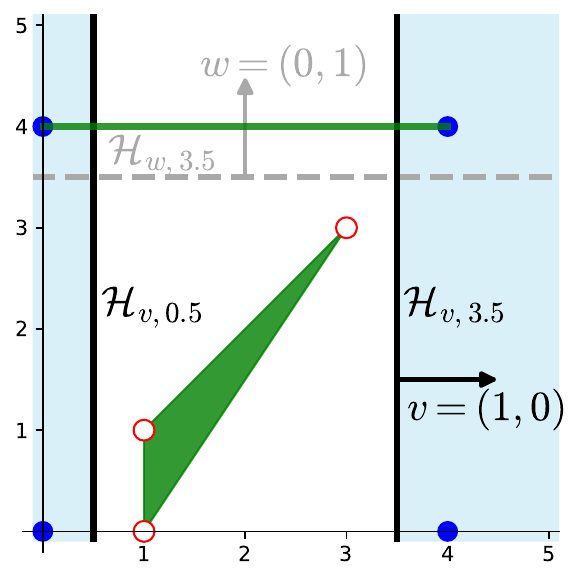}

{\small (a)}
\end{minipage}
\begin{minipage}[h]{0.3\textwidth}
\centering
\includegraphics[scale=0.5]{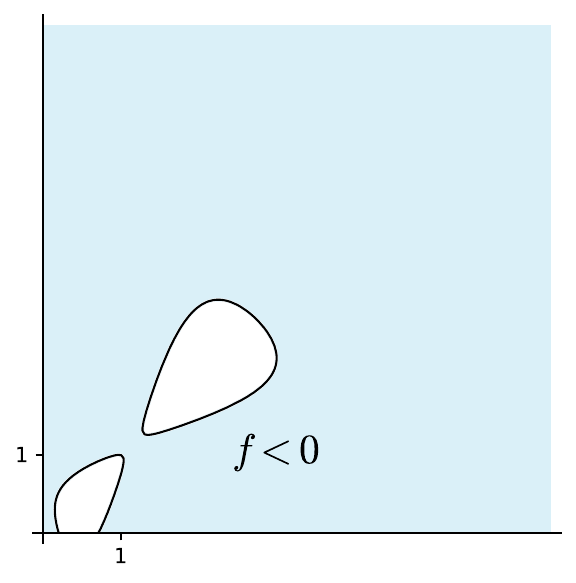}

{\small (b)}
\end{minipage}
\begin{minipage}[h]{0.3\textwidth}
\centering
\includegraphics[scale=0.5]{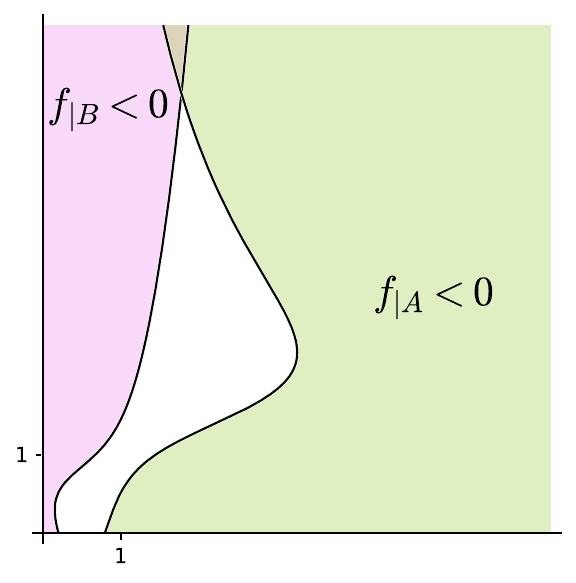}

{\small (c)}
\end{minipage}
\caption{{\small Illustration of Example \ref{Ex_Box} (a) Negative and positive exponent vectors of $f =-x^{4} y^{4} + 10 \, x^{3} y^{3} - 10 \, x^{4} - 10 \, y^{4} + 7 \, x y + 5 \, x - 1$, blue dots are negative, red circles are positive. The black solid lines are strict enclosing hyperplanes of $\sigma_+(f)$. The gray dashed line separates $\Conv((0,4),(4,4))$ from $\Conv(\sigma_+(f))$. (b) Negative connected component of $f$. (c) Negative connected component of $f_{|A} =-x^{4} y^{4} + 10 \, x^{3} y^{3} - 10 \, x^{4}  + 7 \, x y + 5 \, x  $ and $f_{|B}=10 \, x^{3} y^{3}  - 10 \, y^{4} + 7 \, x y + 5 \, x - 1$.}}\label{FIG3}
\end{figure}

By Theorem \ref{Thm::FirstDescartes}(iii), $f^{-1}(\mathbb{R}_{<0})$ is either empty or logarithmically convex, if the signomial $f$ has at most one negative coefficient. In particular, in that case $f$ has one negative connected component. A natural question to ask is what happens for signomials with at most one positive coefficient. For a univariate signomial $f$, it is easy to construct examples where $f$ has one positive coefficient and $f^{-1}(\mathbb{R}_{<0})$ is disconnected, take for example $f = -x^2 + 3x-1$. 

If $n \geq 2$ and the positive exponent vector lies in the interior of the Newton polytope, then $f^{-1}(\mathbb{R}_{<0})$ is homeomorphic to the complement of a bounded convex set  \cite[Corollary 3.5]{telek2023real} \cite[Theorem 3.4]{DescartesHypPlane}; therefore,  $f^{-1}(\mathbb{R}_{<0})$ is connected. This argument does not work if the positive exponent vector lies on the boundary of $\N(f)$, but the conclusion is still true. 

\begin{cor}
\label{Cor_OnePos}
Let $f\colon \mathbb{R}^{n}_{>0} \to \mathbb{R}, \quad x \mapsto \sum_{\mu \in \sigma(f)} c_{\mu}x^{\mu}$  be a signomial. If $\dim \N(f) \geq 2$ and $\# \sigma_+(f) = 1$, then $f^{-1}(\mathbb{R}_{<0} )$ is non-empty and connected.
\end{cor}

\begin{proof}
Write $\{ \alpha \} = \sigma_+(f)$. Since $\dim \N(f) \geq 2$, there exist $\beta_1,\beta_2 \in \sigma_-(f)$ such that $\beta_1,\beta_2$ and $\alpha$ do not lie on a line. In that case, 
\[ \Conv( \{ \beta_1,\beta_2 \}) \cap \Conv( \{ \alpha \}) = \emptyset.\]

Pick a hyperplane $\mathcal{H}_{v,a}$ that contains $\alpha$ such that $\beta_1,\beta_2$ lie in different open half-spaces determined by $\mathcal{H}_{v,a}$. Thus, $(\mathcal{H}_{v,a},\mathcal{H}_{v,a})$ is a pair of strict enclosing hyperplanes of $\sigma_+(f)$.

Using Theorem \ref{Thm_Box}, we conclude that $f^{-1}(\mathbb{R}_{<0} )$ is non-empty and connected.
\end{proof}

Corollary \ref{Cor_OnePos} shows that one can flip the signs in Theorem \ref{Thm::FirstDescartes}(iii) if $\dim \N(f) \geq 2$, i.e. if $f$ has one negative coefficient, then both $f^{-1}(\mathbb{R}_{<0})$ and $(-f)^{-1}(\mathbb{R}_{<0})$ are (possibly empty) connected sets. As discussed at the end of Section \ref{Sec::Background}, a signomial has at most one negative connected component if its positive and negative exponent vectors are separated by a simplex and its negative vertex cones \cite[Theorem 4.6]{DescartesHypPlane}. We proceed by recalling this statement and show that it is possible to flip the signs also in that case under some mild assumptions. 

First we recall the definition of the \emph{negative vertex cone} of a simplex. For an $n$-simplex $P \subseteq \mathbb{R}^{n}$ with vertices $\mu_0, \dots , \mu_n$, the negative vertex cone of $P$ at the vertex $\mu_k$ is defined as
\[ P^{-,k} := \mu_k + \Cone( \mu_k - \mu_0, \dots , \mu_k - \mu_n) .\]
Thus, $P^{-,k}$ is the cone with apex at $\mu_k$ which is generated by the edges pointing into $\mu_k$. For the union of the negative vertex cones $P^{-,0}, \dots ,P^{-,n}$, we write $P^-$.

\begin{thm} \cite[Theorem 4.6]{DescartesHypPlane}
\label{Thm::VertCones46}
Let $f\colon \mathbb{R}^{n}_{>0} \to \mathbb{R}, \, x \mapsto \sum_{\mu \in \sigma(f)} c_{\mu}x^{\mu}$  be a signomial. If there exists an $n$-simplex $P \subseteq \mathbb{R}^{n}$ such that 
\[ \sigma_{-}(f) \subseteq P, \quad \text{ and } \quad \sigma_{+}(f) \subseteq P^{-},\]
then $f^{-1}(\mathbb{R}_{<0} )$  is either empty or contractible.
\end{thm}  

As another consequence of Theorem \ref{Thm_Box}, we have the following result.

\begin{cor}
\label{Lemma_SimplexVertCones}
Let $n \geq 2$ and let $f\colon \mathbb{R}^{n}_{>0} \to \mathbb{R}, \, x \mapsto \sum_{\mu \in \sigma(f)} c_{\mu}x^{\mu}$ be a signomial. Assume that there exists an $n$-simplex $P \subseteq \mathbb{R}^{n}$ such that 
\[ \sigma_{+}(f) \subseteq P \quad \text{ and } \quad \sigma_{-}(f) \subseteq P^{-}. \]
If $\sigma_-(f) \cap \inte(P^-) \neq \emptyset$, then $f^{-1}(\mathbb{R}_{<0} )$ is non-empty and connected.
\end{cor}

\begin{proof}
Denote $\mu_0, \dots , \mu_n$ the vertices of $P$ and write $P$ as
\begin{align}
\label{Eq_Psimplex}
P = \bigcap_{j=0}^{n}\mathcal{H}^{-}_{v_{j},a_j}
\end{align}
for a choice of normal vectors $v_0, \dots , v_n \in \mathbb{R}^{n}$ and scalars $a_0, \dots, a_n \in \mathbb{R}$ such that $\mu_k$ equals the intersection point of the hyperplanes $\mathcal{H}_{v_j,a_j}, j \in \{0,\dots, n\} \setminus \{k \}$. By \cite[Proposition 4.3]{DescartesHypPlane}, each negative vertex cone has the form: 
\[P^{-,k} = \bigcap_{j=0,j\neq k}^n\mathcal{H}^{+}_{v_{j},a_j}  \qquad k=0,\dots,n.\]
Note that if $\beta \in P^{-,k}$, then $\beta$ is contained in $\mathcal{H}^{-,\circ}_{v_{k},a_k}$. Since $\sigma_+(f) \subseteq P$  and $\sigma_-(f) \subseteq P^-$, we have
\begin{align}
\label{Eq-AlphainP}
v_k \cdot \beta \geq a_k \geq v_k \cdot \alpha  \geq v \cdot \mu_k \geq v_k  \cdot\beta', \quad \text{ and } v_k \cdot \alpha > v_k \cdot \beta'
\end{align}
for all $\beta \in  \sigma_-(f) \setminus P^{-,k}, \, \alpha \in \sigma_+(f), \, \beta' \in  \sigma_-(f) \cap P^{-,k}$,  $k=0,\dots,n$.

Let $\beta_0 \in \sigma_-(f) \cap \inte(P^-)$ and assume without loss of generality that $\beta_0 \in \inte(P^{-,0})$, so
\begin{align}
\label{Eq_Beta0inP0}
a_0 > v_0 \cdot \beta_0 \qquad \text{and} \qquad v_j \cdot \beta_0 > a_j, \quad \text{for } j = 1, \dots , n.
\end{align}

If $ \sigma_-(f) \cap P^{-,1} = \emptyset$, then $\sigma_-(f) \subseteq  \mathcal{H}_{v_1,a_1}^+$. Since $ \beta_0 \in \mathcal{H}_{v_1,a_1}^{+,\circ}$ , $\mathcal{H}_{v_1,a_1}$ is a strict separating hyperplane of $\sigma(f)$, which implies that $f^{-1}(\mathbb{R}_{<0} )$ is non-empty and connected by Theorem \ref{Thm::FirstDescartes}(i).

Assume now that $ \sigma_-(f) \cap P^{-,1} \neq \emptyset$, let $\beta_1 \in\sigma_-(f) \cap P^{-,1}$ and $b_1 \in \mathbb{R}$ such that
\[ v_1 \cdot \alpha > b_1 > v_1 \cdot \beta_1 \qquad \text{for all } \alpha \in \sigma_+(f).\]
From \eqref{Eq-AlphainP}, it follows that $(\mathcal{H}_{v_1,a_1}, \mathcal{H}_{v_1,b_1})$ is a pair of enclosing hyperplanes of $\sigma_+(f)$. By \eqref{Eq_Beta0inP0}, we have $\beta_0 \in \mathcal{H}^{+,\circ}_{v_1,a_1}$. Thus, $(\mathcal{H}_{v_1,a_1}, \mathcal{H}_{v_1,b_1})$ is a pair of strict enclosing hyperplanes.

Every element $\mu \in \Conv(\beta_0 , \beta_1) \setminus \{ \beta_1 \}$ has the form $\mu = t \beta_0 + (1-t) \beta_1$ for some $t \in (0,1]$. By \eqref{Eq-AlphainP} and \eqref{Eq_Beta0inP0}, we have
\[ v_2 \cdot \mu = t (v_2 \cdot \beta_0) + (1- t) (v_2 \cdot \beta_1) > a_2.\] 
Thus, $\Conv(\beta_0 , \beta_1) \setminus \{ \beta_1 \} \subseteq \mathcal{H}_{v_2,a_2}^{+,\circ}$, which implies that 
\[ \Conv(\{ \beta_0,\beta_1 \} ) \cap \Conv(\sigma_+(f)) = \emptyset. \]
From Theorem \ref{Thm_Box}, it follows that $f^{-1}(\mathbb{R}_{<0} )$ is non-empty and connected.
\end{proof}

\begin{ex} 
\label{Ex_Simplex} 
\textbf{(a)} Consider the signomial
\[ f =-x^{5} y^{\frac{7}{3}} - x^{5} y^{2} + x^{2} y^{2} - x y^{3} - y^{4} + 2 \, x^{2} y^{\frac{4}{3}} + 2 \, x y^{2} - x y.\]
The simplex $P=\operatorname{Conv}((1,1),(4,2),(1,3))$ contains $\sigma_+(f)$ and the negative exponent vectors are contained in the union of the negative vertex cones $P^-$. By Theorem \ref{Thm::VertCones46}, we have that $(-f)^{-1}(\mathbb{R}_{<0})$ is connected. From Corollary \ref{Lemma_SimplexVertCones}, it follows that $f^{-1}(\mathbb{R}_{<0})$ is connected as well. The exponent vectors of $f$, the simplex $P$ and its negative vertex cones are depicted in Figure~\ref{FIG4}(a),(b).

To write $P$ as an intersection of half-spaces as in \eqref{Eq_Psimplex} in the proof of Corollary \ref{Lemma_SimplexVertCones}, one can choose $v_0 = (-1,0), \, v_1 = (0.5,1.5), \, v_2 = (0.5,-1.5)$. With this choice we have:
\[ P = \mathcal{H}_{v_0,-1}^- \cap \mathcal{H}_{v_1,5}^- \cap \mathcal{H}_{v_2,-1}^-. \]

\medskip

\textbf{(b)} The following example shows that the assumption $\sigma_-(f) \cap \inte(P^-) \neq \emptyset$ in Corollary \ref{Lemma_SimplexVertCones} is necessary. If we remove the exponent vectors $(0,4),(5,2) \in  \relint(P^-)$ from the support of $f$, the signomial
\[ g =-x^{5} y^{\frac{7}{3}} + x^{2} y^{2} - x y^{3} + 2 \, x^{2} y^{\frac{4}{3}} + 2 \, x y^{2} - x y\]
satisfies that $\sigma_+(g) \subseteq P$ and $\sigma_-(g) \subseteq P^-$; however, $g^{-1}(\mathbb{R}_{<0})$ has two connected components. Figure \ref{FIG4}(c),(d) displays $\sigma(g)$ and $g^{-1}(\mathbb{R}_{<0})$.
\end{ex}

\begin{figure}[t]
\centering
\begin{minipage}[h]{0.45\textwidth}
\centering
\includegraphics[scale=0.5]{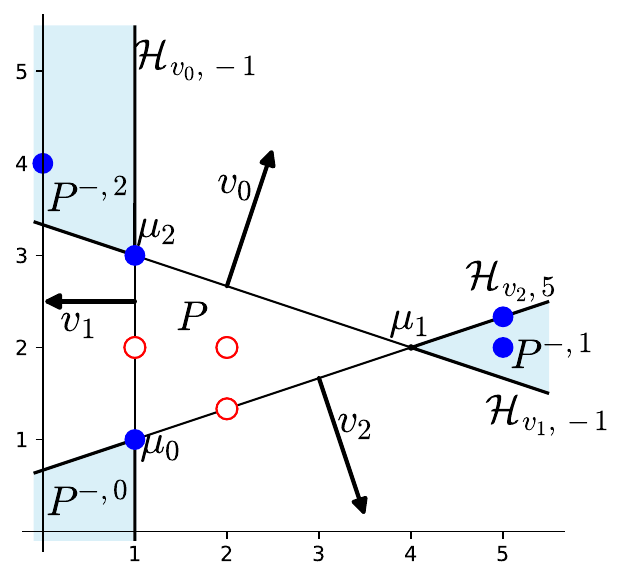}

{\small (a)}
\end{minipage}
\begin{minipage}[h]{0.45\textwidth}
\centering
\includegraphics[scale=0.5]{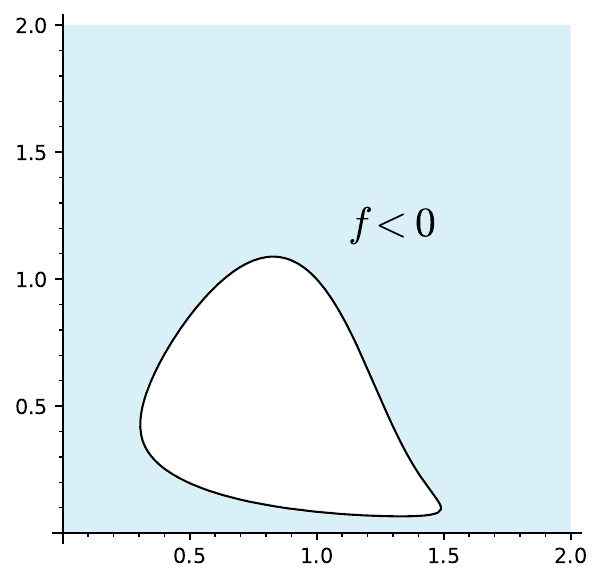}

{\small (b)}
\end{minipage}

\begin{minipage}[h]{0.45\textwidth}
\centering
\includegraphics[scale=0.5]{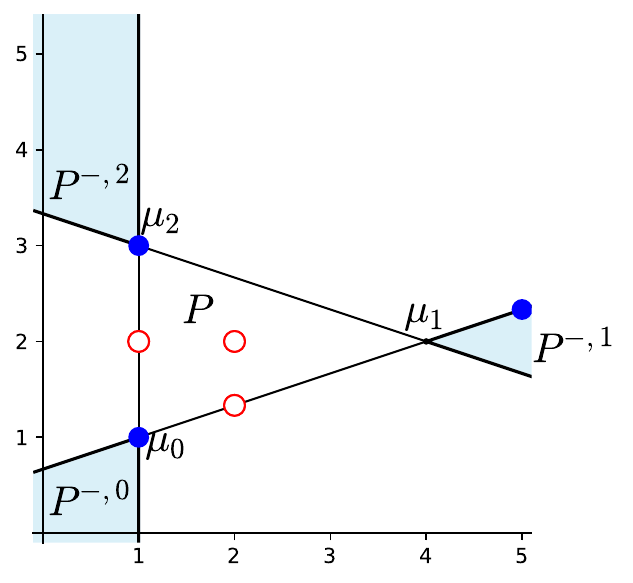}

{\small (c)}
\end{minipage}
\begin{minipage}[h]{0.45\textwidth}
\centering
\includegraphics[scale=0.5]{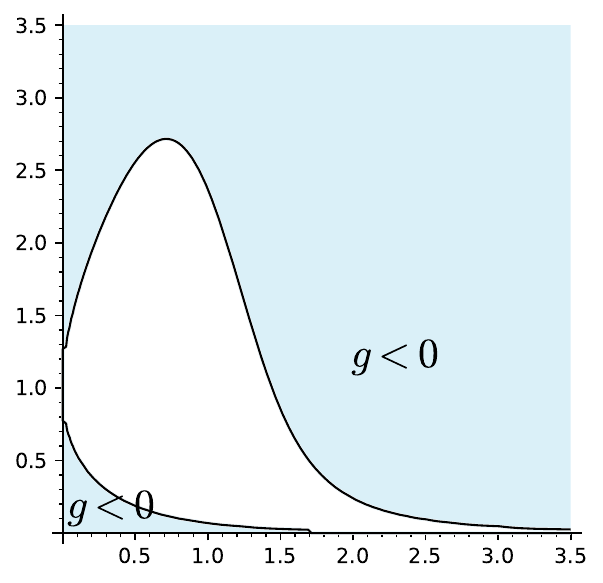}

{\small (d)}
\end{minipage}
\caption{{\small Illustration of Example \ref{Ex_Simplex} (a) A simplex $P$, its negative vertex cones, and the support of $f = -x^{5} y^{\frac{7}{3}} - x^{5} y^{2} + x^{2} y^{2} - x y^{3} - y^{4} + 2 \, x^{2} y^{\frac{4}{3}} + 2 \, x y^{2} - x y$. (c) The support of $g=-x^{5} y^{\frac{7}{3}} + x^{2} y^{2} - x y^{3} + 2 \, x^{2} y^{\frac{4}{3}} + 2 \, x y^{2} - x y$.}}\label{FIG4}
\end{figure}

\section{Reduction to faces of the Newton polytope}
\label{Sec_RedFaces}

\subsection{Negative and parallel faces}
\label{Section::NegAndParaFaces}

 In this section, we present two criteria to reduce the problem of finding the number of negative connected components of a signomial to the same problem for a signomial in less variables and monomials. The approach is based on how the exponent vectors of the signomial lie on the faces of the Newton polytope. A \emph{face} of the Newton polytope of a signomial $f$ is a set of the form
\[ \N(f)_v := \big\{ \omega \in \N(f) \mid v \cdot  \omega = \max_{\mu \in \N(f)} v \cdot \mu \big\}\]
for some $v \in \mathbb{R}^{n}$. The vector $v$ is called \emph{outer normal vector} of the face $\N(f)_v$. A polytope has finitely many faces \cite[Theorem 3.46]{JoswigTheobald_book}.  For a fixed face $F \subseteq \N(f)$, the set of vectors $v \in \mathbb{R}^{n}$ such that $\N(f)_v = F$ is called the \emph{outer normal cone} of $F$. For more details about polytopes and their faces, we refer to the books \cite{Ziegler_book, BasicsOnPolytopes, JoswigTheobald_book}. To compute faces and outer normal cones of polytopes, one can use e.g. \texttt{Polymake} \cite{polymake} or \texttt{SageMath} \cite{sagemath}.

\begin{thm}
\label{Thm_NegFace}
Let $f\colon \mathbb{R}^{n}_{>0} \to \mathbb{R}, \, x \mapsto \sum_{\mu \in \sigma(f)} c_{\mu}x^{\mu}$ be a signomial.
If there exists a face $F \subseteq \N(f)$ such that $\sigma_{-}(f) \subseteq F$, then
\[b_0\big(f^{-1}(\mathbb{R}_{<0})\big) = b_0\big(f_{|F}^{-1}(\mathbb{R}_{<0})\big).\]
\end{thm}

\begin{proof}
It follows directly from Proposition \ref{Prop_NonStrictSepHyp} with $R = \sigma(f) \cap F$.
\end{proof}

\begin{ex}
\label{Ex_NegFace}
Consider the signomial
\[f_{|B}(x,y) = 50  x^{2} y^{3} + x y^{3} + y^{4} - 9.5  y^{3} + 51  x^{2} + 30.5  y^{2} - 37  y + 12\]
from Example \ref{Ex::ExRunning} whose Newton polytope is shown in Figure \ref{FIG1}(c). 

The face $F = \Conv\big( (0,0),(0,4)\big) \subseteq \N(f_{|B})$ contains all the negative exponent vectors of $f_{|B}$. Since $f_{|F}$ is univariate, it is easy to conclude that $f_{|F}^{-1}(\mathbb{R}_{<0})$ has two connected components. By Theorem \ref{Thm_NegFace}, $f_{|B}^{-1}(\mathbb{R}_{<0})$ has also two connected components.
\end{ex}

\begin{figure}[t]
\centering
\centering
\includegraphics[scale=0.5]{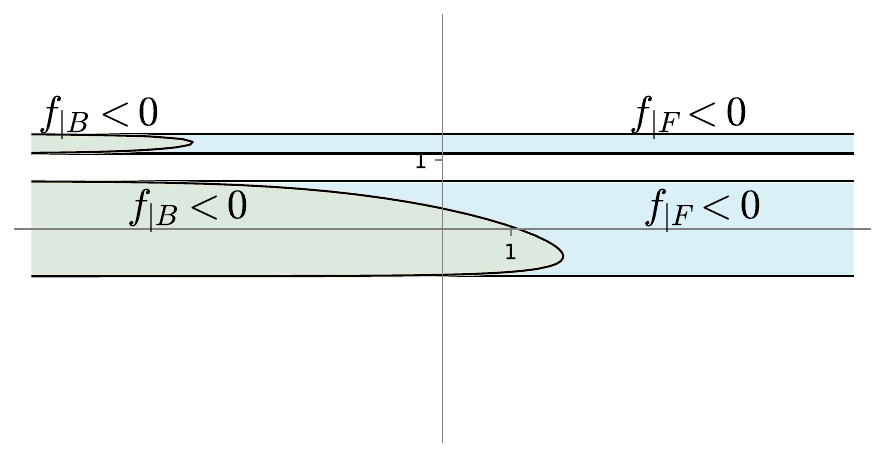}
\caption{{\small Images of the negative connected components of $f_{|B}$ and $f_{|F}$ from Example \ref{Ex_NegFace} under the coordinate-wise natural logarithm map. }}\label{FIG6}
\end{figure}

The next proposition refines the bound from Proposition \ref{Prop_StripBound0} in the special case where the pair of enclosing hyperplanes $(\mathcal{H}_{v,a},\mathcal{H}_{v,b})$ of $\sigma_+(f)$ is non-strict and all the exponent vectors of $f$ lie on the union of the two hyperplanes. In this case, the hyperplanes $\mathcal{H}_{v,a},\mathcal{H}_{v,b}$ cut out two parallel faces of the Newton polytope of $f$, i.e.:
\[ \N(f)_v = \N(f) \cap \mathcal{H}_{v,a}, \quad \N(f)_{-v} = \N(f) \cap \mathcal{H}_{v,b}.\]

Similarly to \eqref{Eq::GraphABVertex}, we define a bipartite graph whose set of vertices and edges are
\begin{equation}
\label{Eq::GraphvVertex}
\begin{aligned}
\mathcal{B}_v &:= \mathcal{B}_0^-(f_{|\N(f)_v } ) \sqcup \mathcal{B}_0^-(f_{|\N(f)_{-v}}), \\
\mathcal{E}_v &:= \{(U,V) \mid U \in \mathcal{B}_0^-(f_{|\N(f)_v } ), V \in \mathcal{B}_0^-(f_{|\N(f)_{-v}}) \colon U \cap V \neq \emptyset \}.
\end{aligned}
\end{equation}

\begin{prop}
\label{Thm_ParallelFaces}
Let $f\colon \mathbb{R}^{n}_{>0} \to \mathbb{R}, \, x \mapsto \sum_{\mu \in \sigma(f)} c_{\mu}x^{\mu}$ be a signomial. Assume that there exists  $v \in \mathbb{R}^{n}$ such that $\sigma(f) \subseteq \N(f)_v \cup \N(f)_{-v}$. Then
\[ b_0\big( f^{-1} (\mathbb{R}_{<0}) \big) = D \leq b_0\big( f_{|N(f)_v}^{-1} (\mathbb{R}_{<0}) \big) + b_0\big( f_{|N(f)_{-v}}^{-1} (\mathbb{R}_{<0}) \big),\]
where $D$ denotes the number of the connected components of the graph $(\mathcal{B}_v,\mathcal{E}_v)$ from \eqref{Eq::GraphvVertex}.
\end{prop}

\begin{proof}

After an affine change of coordinates (see e.g. \cite[Lemma 2.3]{DescartesHypPlane}), we assume that $v = (0, \dots , 0 ,1) \in \mathbb{R}^n$, $\N(f)_v = \N(f) \cap \mathcal{H}_{v,1}$ and $\N(f)_{-v} = \N(f) \cap \mathcal{H}_{v,0}$. There exist signomials $f_1,f_2 \colon \mathbb{R}^{n-1}_{>0} \to \mathbb{R}$ such that  $f(x) = f_1(x_1, \dots , x_{n-1}) + x_n f_2(x_1, \dots , x_{n-1})$ for all $x \in \mathbb{R}^n_{>0}$. Note that $f_1^{-1}(\mathbb{R}_{<0})$ (resp. $f_2^{-1}(\mathbb{R}_{<0})$) has the same number of connected components as $f_{|N(f)_{-v}}^{-1}(\mathbb{R}_{<0})$ (resp. $f_{|N(f)_{v}}^{-1}(\mathbb{R}_{<0})$). Moreover, the number of connected components of $f_1^{-1}(\mathbb{R}_{<0})\cup f_2^{-1}(\mathbb{R}_{<0})$ is $D$. We decompose this set into three pairwise disjoint sets
\begin{align*}
 W_0 &=  \{ \tilde{x} \in \mathbb{R}^{n-1}_{>0} \mid f_1(\tilde{x}) < 0,\, f_2(\tilde{x}) \leq 0\} \cup  \{  \tilde{x} \in \mathbb{R}^{n-1}_{>0} \mid  f_1(\tilde{x}) \leq 0,\, f_2(\tilde{x}) < 0\}\\
 W_1 &=  \{\tilde{x} \in \mathbb{R}^{n-1}_{>0} \mid   f_1(\tilde{x}) < 0,\, f_2(\tilde{x}) > 0\}, \quad W_2 =  \{\tilde{x} \in \mathbb{R}^{n-1}_{>0} \mid   f_1(\tilde{x}) > 0,\, f_2(\tilde{x}) < 0\}.
 \end{align*}
Consider the map 
\[ F \colon \mathbb{R}^{n-1}_{>0} \to \mathbb{R},  \quad \tilde{x} \mapsto -\tfrac{f_1(\tilde{x})}{f_2(\tilde{x})} \]
which is continous on $W_1 \cup W_2$ and has some poles and zeros in $W_0$. It holds that
\begin{equation}
\label{Eq::Elevator}
\begin{aligned}
&\text{ if }  \tilde{x} \in W_0, \text{ then} &(\tilde{x},x_n) \in f^{-1}(\mathbb{R}_{<0}) \quad  &\text{ for all  } x_n > 0, \\
&\text{ if }  \tilde{x} \in W_1, \text{ then}  &(\tilde{x},x_n) \in f^{-1}(\mathbb{R}_{<0}) \quad  &\text{ if and only if  }  x_n < F(\tilde{x}),  \\
& \text{ if }  \tilde{x} \in W_2, \text{ then}  &(\tilde{x},x_n) \in f^{-1}(\mathbb{R}_{<0}) \quad  &\text{ if and only if  }  x_n > F(\tilde{x}). 
\end{aligned}
\end{equation}
 
It follows that  the projection map
  \[\pi\colon f^{-1}(\mathbb{R}_{<0}) \to  f_1^{-1}(\mathbb{R}_{<0})\cup f_2^{-1}(\mathbb{R}_{<0}), \quad  \, x \mapsto \pi(x) = (x_1, \dots , x_{n-1}) \]
is well defined, continuous and surjective. This gives that the number of connected components of $f^{-1}(\mathbb{R}_{<0}) $ is at most $D$.

To show equality, we show that for every connected component $U$ of $f_1^{-1}(\mathbb{R}_{<0})\cup f_2^{-1}(\mathbb{R}_{<0})$, the set $\pi^{-1}(U)$ is connected. For $i=0,1,2$, let $U_{i,1},\dots ,U_{i,k_i}$ be the connected components of $U \cap W_i$. By \eqref{Eq::Elevator}, the sets $\pi^{-1}(U_{i,j})$ are connected. Therefore it is enough to show that $\pi^{-1}(U_{i,j} \cup U_{i',j'})$ is connected if $U_{i,j} \cup U_{i',j'}$ is connected for $i,i' \in \{0,1,2\}$, $j \in \{1,\dots,k_i\},\, j' \in \{1,\dots,k_{i'}\}$ . Note that  $U_{1,j} \cup U_{2,j'}$ cannot be connected through a path that is contained in $f_1^{-1}(\mathbb{R}_{<0})\cup f_2^{-1}(\mathbb{R}_{<0})$. Using \eqref{Eq::Elevator}, it follows that $\pi^{-1}(U_{0,j} \cup U_{i',j'})$, $i' =1,2$ is connected. This concludes the proof. For an illustration, we refer to Figure \ref{FIG4_new} and Example \ref{Ex:NewParaFaces}.
\end{proof}

\begin{ex}
\label{Ex:NewParaFaces}
To illustrate the proof of Proposition \ref{Thm_ParallelFaces}, we consider the signomials
\[ f_1(x_1) = x_1^{4} - 16 \, x_1^{3} + 83 \, x_1^{2} - 164 \, x_1 + 96, \quad f_2(x_1) = x_1^{2} - 10 \, x_1 + 20, \quad f(x_1,x_2) = f_1(x_1) + x_2 f_2(x_1).\]
The exponent vectors of $f$ and its Newton polytope is shown in Figure \ref{FIG4_new}(a). The sets $W_0, W_1, W_2$ from the proof of Proposition \ref{Thm_ParallelFaces} have the form
\begin{align}
\label{Eq::NewParaFacesW}
 W_0 = (\xi_1,3) \cup (4,\xi_2), \qquad W_1 = (1,\xi_1)\cup(\xi_2,8), \qquad W_2 = (3,4) 
\end{align}
where $\xi_1 = 5 -\sqrt{5} \approx 2.76, \, \xi_2 = 5 +\sqrt{5} 	\approx 7.23$ are the roots of $f_2$. The preimages of $W_0,W_1,W_2$ under the projection map $\pi\colon \mathbb{R}_{>0}^2 \to \mathbb{R}_{>0}, \, (x_1,x_2) \mapsto x_1$ are depicted in Figure \ref{FIG4_new}(b). Since $W_0\cup W_1 \cup W_2$ is a connected set, it follows that $f$ has one negative connected component.
\end{ex}

\begin{figure}[t]
\centering
\begin{minipage}[h]{0.45\textwidth}
\centering
\includegraphics[scale=0.6]{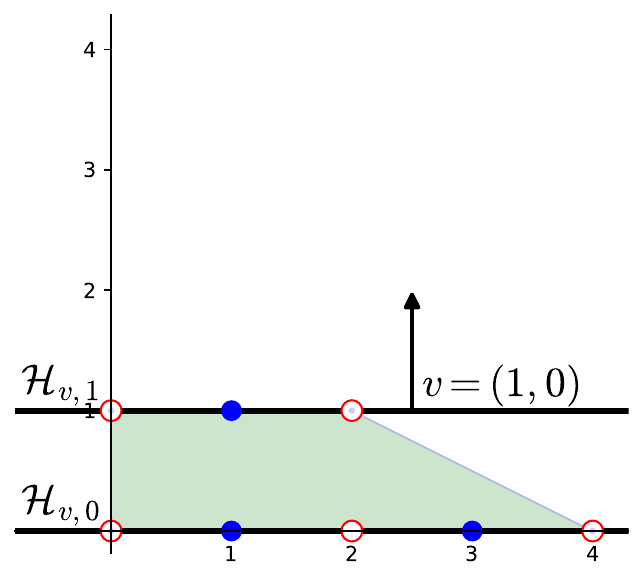}

{\small (a)}
\end{minipage}
\hspace{10pt}
\begin{minipage}[h]{0.45\textwidth}
\centering
\includegraphics[scale=0.7]{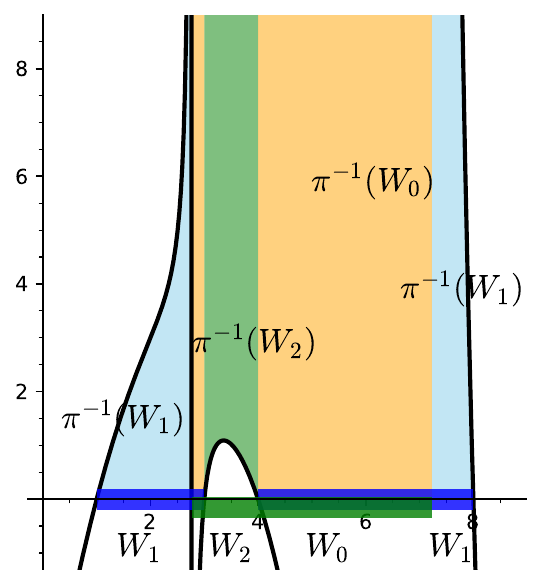}

{\small (b)}
\end{minipage}

\caption{{\small Exponent vectors and the Newton polytope of $f$ from Example~\ref{Ex:NewParaFaces}.  (b) Preimages of the sets $W_0,W_1,W_2$ from \eqref{Eq::NewParaFacesW} under the coordinate projection $\pi $.  }}\label{FIG4_new}
\end{figure}

\begin{remark}
Let $f$ be a signomial satisfying
\[ \sigma(f) \subseteq \N(f)_v \cup \N(f)_{-v}\]
for some $v \in \mathbb{R}$ as in Proposition \ref{Thm_ParallelFaces} and let $A,B$ the sets as in \eqref{Eq::DefAB}. By Theorem \ref{Thm_NegFace}, there is a bijection between the sets
\[ \mathcal{B}_0^-(f_{|\N(f)_v} )\sqcup  \mathcal{B}_0^-(f_{|\N(f)_{-v}} ) \longleftrightarrow \mathcal{B}_0^-(f_{|A} ) \sqcup \mathcal{B}_0^-(f_{B}) .\]
Thus, the graphs $(\mathcal{B}_{A,B},\mathcal{E}_{A,B})$ from \eqref{Eq::GraphABVertex} and  $(\mathcal{B}_{v},\mathcal{E}_{v})$ from \eqref{Eq::GraphvVertex} have the same number of vertices. Since
\[ f_{|A}^{-1}(\mathbb{R}_{<0}) \subseteq f_{|\N(f)_v}^{-1}(\mathbb{R}_{<0}) \quad \text{and} \quad  f_{|B}^{-1}(\mathbb{R}_{<0})  \subseteq f_{|\N(f)_{-v}}^{-1}(\mathbb{R}_{<0}),\]
every edge of the graph $(\mathcal{B}_{A,B},\mathcal{E}_{A,B})$ corresponds to an edge of the graph $(\mathcal{B}_{v},\mathcal{E}_{v})$. Thus, the bound provided in Proposition \ref{Thm_ParallelFaces} is always smaller or equal than the bound given in Proposition~\ref{Lema_StripBound}.
\end{remark}

\begin{ex}
\label{Ex_Running}
 The bound on $b_0(f^{-1}(\mathbb{R}_{<0}))$ in Proposition \ref{Lema_StripBound} and in Proposition \ref{Thm_ParallelFaces} can be different, even though the two statements look similar.

To demonstrate this, consider the signomial $f= 73 \, x  - 55 \, x^{2}  - x^{4}  + y - 20 \, x y + x^{4} y$. The Newton polytope of $f$ is shown in Figure \ref{FIG7}(a). We have that  $\sigma(f) \subseteq \N(f)_v \cup \N(f)_{-v}$ for $v = (0,1)$. The restrictions of $f$ to these two faces are given by:
\[  f_{\N(f)_v} =  73 \, x  - 55 \, x^{2}  - x^{4}  , \qquad f_{\N(f)_{-v}} = y - 20 \, x y + x^{4} y.\]
The pair $(\mathcal{H}_{v,1}, \, \mathcal{H}_{v,0})$ is a pair of enclosing hyperplanes of $\sigma_+(f)$. Let $A,B \subseteq \sigma(f)$ be as defined in \eqref{Eq::DefAB}:
\begin{align*}
A &= (\mathcal{H}^{+}_{v,1} \cap \sigma_-(f)) \cup \sigma_+(f) =  \{(1,1), (1,0),(0,1),(4,1)\}, \\ 
B &= (\mathcal{H}^{-}_{v,0} \cap \sigma_-(f)) \cup \sigma_+(f) =  \{(2,0),(4,0), (1,0),(0,1),(4,1)\}.
\end{align*}
Since $f_{|\N(f)_v}$ and $f_{|A}$ only have one negative coefficient, the sets  $f_{|\N(f)_v}^{-1}(\mathbb{R}_{<0})$ and $f_{|A}^{-1}(\mathbb{R}_{<0})$ are connected by Theorem \ref{Thm::FirstDescartes}(iii). Since the supports of $f_{|\N(f)_{-v}}$ and $f_{|B}$ have strict separating hyperplanes, Theorem \ref{Thm::FirstDescartes}(i) implies that $f_{|\N(f)_{-v}}^{-1}(\mathbb{R}_{<0})$, $f_{|B}^{-1}(\mathbb{R}_{<0})$ are connected. 

One can also verify that  $f_{|A}^{-1}(\mathbb{R}_{<0}) \cap  f_{|B}^{-1}(\mathbb{R}_{<0}) = \emptyset$, e.g. using the \texttt{Maple} \cite{maple} function \texttt{IsEmpty()}. Thus, the bound on the number of connected components of $f^{-1}(\mathbb{R}_{<0})$ provided by Proposition \ref{Lema_StripBound} is two.  The negative connected components of $f_{|A}, \, f_{|B}$ are shown in Figure~\ref{FIG7}(c).

On the other hand, $(1,1) \in f_{|\N(f)_{v}}^{-1}(\mathbb{R}_{<0}) \cap f_{|\N(f)_{-v}}^{-1}(\mathbb{R}_{<0})$. Thus, Proposition \ref{Thm_ParallelFaces} gives that $f^{-1}(\mathbb{R}_{<0})$ is connected.
\end{ex}

\begin{figure}[t]
\centering
\begin{minipage}[h]{0.3\textwidth}
\centering
\includegraphics[scale=0.4]{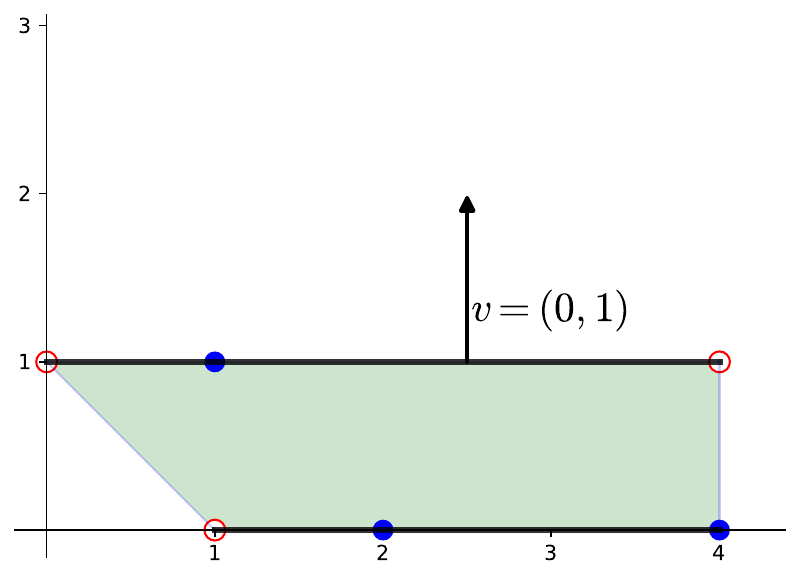}

{\small (a)}
\end{minipage}
\begin{minipage}[h]{0.3\textwidth}
\centering
\includegraphics[scale=0.4]{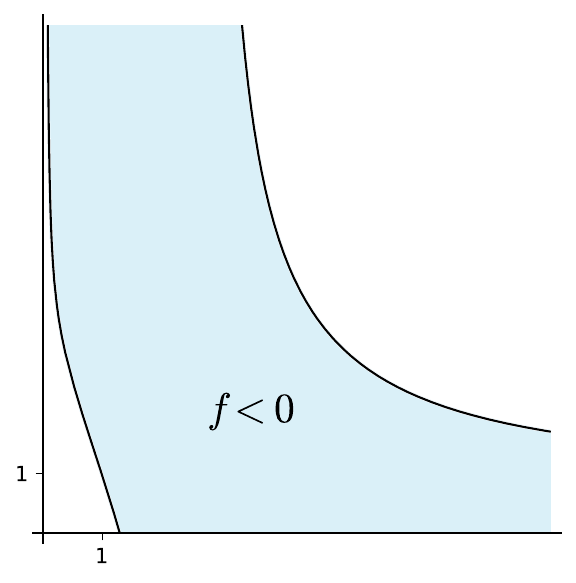}

{\small (b)}
\end{minipage}
\begin{minipage}[h]{0.3\textwidth}
\centering
\includegraphics[scale=0.4]{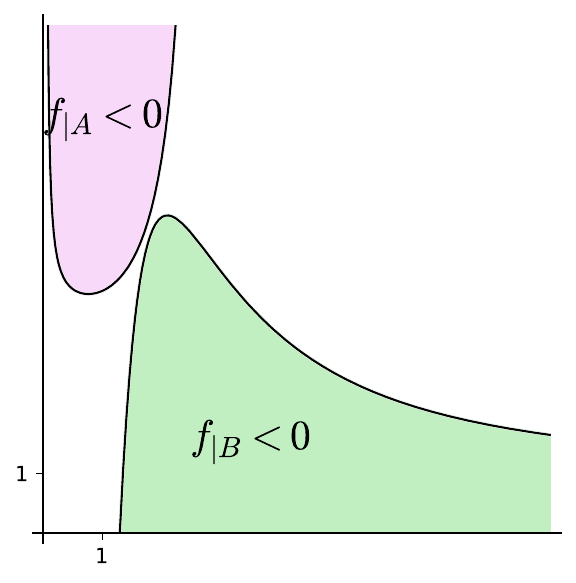}

{\small (c)}
\end{minipage}
\caption{{\small Illustration of Example \ref{Ex_Running} (a)  Newton polytope of $f= 73 \, x  - 55 \, x^{2}  - x^{4}  + y - 20 \, x y + x^{4} y$. (b) The negative connected component of  $f$. (c) Negative connected components of $f_{|A}$ (purple) and $f_{|B}$ (green), where $A =  \{(1,1), (1,0),(0,1),(4,1)\}, \, B =  \{(2,0),(4,0), (1,0),(0,1),(4,1)\}$.}}\label{FIG7}
\end{figure}

From Proposition \ref{Thm_ParallelFaces}, we derive a criterion to ensure that a polynomial has at most one negative connected component. This criterion can be interpreted as a version of Theorem \ref{Thm_Box} where we replace strict separating hyperplanes by non-strict ones. 

\begin{thm}
\label{Thm_ParallelFacesEdge}
Let $f\colon \mathbb{R}^{n}_{>0} \to \mathbb{R}, \, x \mapsto \sum_{\mu \in \sigma(f)} c_{\mu}x^{\mu}$ be a signomial. Assume that there exists $v \in \mathbb{R}^{n}$ such that $\sigma(f) \subseteq \N(f)_v \cup \N(f)_{-v}$ and 
\[b_0 \big( f_{|\N(f)_{v}}^{-1}(\mathbb{R}_{<0}) \big) = b_0 \big( f_{|\N(f)_{-v}}^{-1}(\mathbb{R}_{<0}) \big)   = 1.\]
If there exist negative exponent vectors $\beta_1 \in \N(f)_v$ and  $\beta_2 \in \N(f)_{-v}$ such that $\Conv(\beta_1, \beta_2)$ is an edge of $\N(f)$, then $b_0 \big( f^{-1}(\mathbb{R}_{<0}) \big) = 1$.
\end{thm}

\begin{proof}
Throughout the proof we assume that $\N(f)_v \neq \N(f)_{-v}$, otherwise the statement is obvious. First, we show that $\Conv(\beta_1, \beta_2) \cap \sigma(f) = \{ \beta_1, \beta_2 \} .$ For $\alpha \in \Conv(\beta_1, \beta_2) \cap \sigma(f)$, if $\alpha \notin  \{ \beta_1, \beta_2 \}$, then there exists $t \in (0,1)$ such that $\alpha = t \beta_1 + (1-t) \beta_2$. Let $\mathcal{H}_{v,a}$ and $\mathcal{H}_{v,b}$ be the supporting hyperplanes of $\N(f)_v$ and $\N(f)_{-v}$ respectively. As $\N(f)_v \neq \N(f)_{-v}$, $a > b$. Thus, we have 
\[ v \cdot \alpha = t (v \cdot \beta_1) + (1-t) (v \cdot \beta_2) = ta + (1-t)b \neq a,b  \qquad \text{ if } t \in (0,1).\]
which contradicts $\alpha \in \mathcal{H}_{v,a} \cup \mathcal{H}_{v,b}$.

Proposition \ref{Prop_Box} implies that $f^{-1}_{|\N(f)_{v}} (\mathbb{R}_{<0}) \cap f^{-1}_{|\N(f)_{-v}} (\mathbb{R}_{<0}) \neq \emptyset$. Thus, by Proposition \ref{Thm_ParallelFaces} we have $b_0\big( f^{-1}(\mathbb{R}_{<0}) \big) = 1$.
\end{proof}

\begin{remark}
In the last step of the proof of Theorem \ref{Thm_ParallelFacesEdge},  instead of Proposition \ref{Thm_ParallelFaces} one could use Theorem \ref{Thm_NegFace} and Proposition \ref{Lema_StripBound} as well to conclude that  $b_0\big( f^{-1}(\mathbb{R}_{<0}) \big) = 1$.
\end{remark}

\begin{ex}
\label{Ex_cube}
For each $c_1, \dots , c_8 \in \mathbb{R}_{>0}$, the Newton polytope of 
\[f(x,y,z) = c_1 x+c_2 xy- c_3y-c_4+c_5 yz- c_6z- c_7 xz- c_8 xyz\]
 is the cube depicted in Figure \ref{FIG8}. The top and the bottom faces of the cube are parallel to each other and contain all the exponent vectors of $f$. Choosing $v = (0,0,1)$, we have that $\sigma(f) \subseteq \N(f)_v \cup \N(f)_{-v}$. Since both $\sigma(f_{|\N(f)_v})$ and $\sigma(f_{|\N(f)_{-v}})$ have a strict separating hyperplane, $b_0 \big( f_{|\N(f)_{v}}^{-1}(\mathbb{R}_{<0}) \big) = b_0 \big( f_{|\N(f)_{-v}}^{-1}(\mathbb{R}_{<0}) \big)   = 1$ by Theorem \ref{Thm::FirstDescartes}(i).

The vertices of the edge $\Conv\big( (0,0,0),(0,0,1) \big)$  correspond to negative exponent vectors, thus $f^{-1}(\mathbb{R}_{<0})$ is connected by Theorem \ref{Thm_ParallelFacesEdge}.
\end{ex}

\begin{figure}[t]
\centering
\includegraphics[scale=0.5]{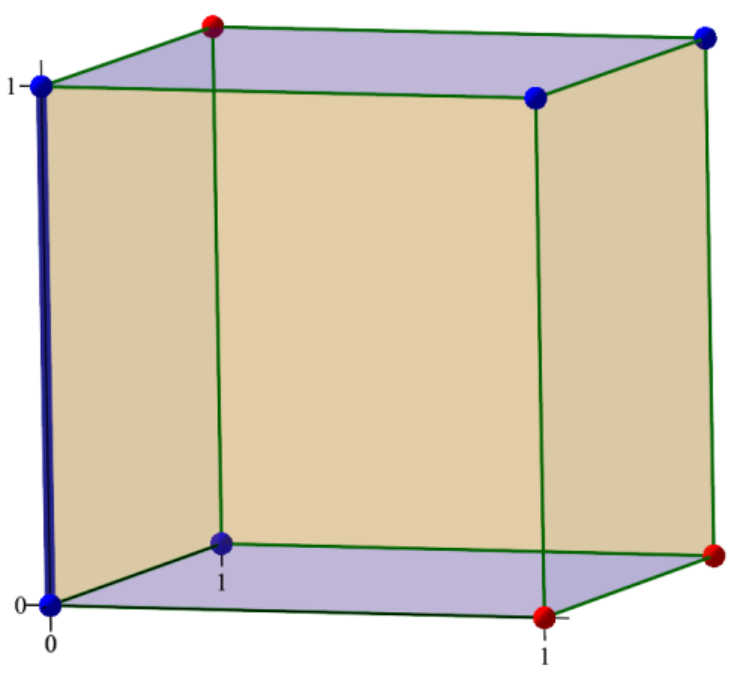}
\caption{{\small Newton polytope of $f(x,y,z) = c_1 x+c_2 xy- c_3y-c_4 +c_5 yz- c_6z- c_7 xz- c_8 xyz$ from Example \ref{Ex_cube}. Red (resp. blue) dots correspond to positive (resp. negative) exponent vectors. The blue squares are parallel faces of $\N(f)$ whose union contains $\sigma(f)$. The blue thick edge $\Conv( (0,0,0),(0,0,1) )$ joins the two blue parallel faces and its vertices are negative exponent vectors. }}\label{FIG8}
\end{figure}

\subsection{Algorithm for connectivity}
\label{Sec_Algo}

Based on Theorem \ref{Thm_NegFace} and Theorem \ref{Thm_ParallelFacesEdge}, we give a recursive algorithm that checks a sufficient condition for ensuring that a signomial $f$ has one negative connected component. Using Theorem \ref{Thm_NegFace}, we reduce $f$ to a face of its Newton polytope that contains all the negative exponent vectors if possible. Using Theorem \ref{Thm_ParallelFacesEdge}, we split up $f$ to parallel faces of its Newton polytope whose union contains all exponent vectors of $f$. We repeat this reduction until the polynomials are simple enough and we can apply one of the following criterion:
\begin{itemize}
\item[(i)] $f$ has one negative coefficient, Theorem \ref{Thm::FirstDescartes}(iii),
\item[(ii)] $f$ has one positive coefficient and $n \geq 2$, Corollary \ref{Cor_OnePos},
\item[(iii)] the support of $f$ has a strict separating hyperplane, Theorem  \ref{Thm::FirstDescartes}(i),
\item[(iv)] the exponent vectors of $f$ lie in and outside of a simplex as in Theorem \ref{Thm::VertCones46} or in Corollary \ref{Lemma_SimplexVertCones}.
\end{itemize}
We define a submethod \texttt{CheckConnectivity(}\texttt{)} that checks these conditions. If one of (i)-(iv) holds, \texttt{CheckConnectivity(}$f$\texttt{)}  returns true and we know that $f^{-1}(\mathbb{R}_{<0})$ is connected. If none of the conditions (i)-(iv) is true, \texttt{CheckConnectivity(}$f$\texttt{)}  returns false. Checking (i) and (ii) is simple. The problem of deciding whether $\sigma(f)$ has a strict separating hyperplane can be reduced to a feasibility problem in linear programming. Such a problem can be solved in polynomial time and there exist efficient algorithms that work well in practice  \cite[Chapter 3]{Lovasz}. Therefore, condition (iii) can be checked even for signomials in many variables and many monomials. Checking condition (iv) is a significantly harder problem, and in practice we might avoid it.

The submethod \texttt{IntersectionNonempty()} checks whether there exists an edge of the Newton polytope between two parallel faces such that both vertices of the edge correspond to a negative exponent vector.
\begin{algorithm}
\caption{\texttt{CheckConnectivityRecursive} }
\label{Algo_Conn}
\begin{algorithmic}[1]
\REQUIRE a signomial $f$
\ENSURE \TRUE \, if $f^{-1}(\mathbb{R}_{<0})$ is connected, \FALSE  \, if the method is inconclusive

\IF{\texttt{CheckConnectivity(}$f$\texttt{)} = \TRUE }
\RETURN \TRUE
\ENDIF

\STATE $F \leftarrow$ smallest face of $\N(f)$ that contains $\sigma_{-}(f)$

\IF{$F \subseteq \N(f)$ is a proper face }
\RETURN \texttt{CheckConnectivityRecursive(} $f_{|F}$\texttt{)}
\ENDIF

\FOR{ every proper face $\N(f)_{v}  \subseteq \N(f)$ such that $\sigma(f) \subseteq \N(f)_{v} \cup \N(f)_{-v} $ }
\IF{ \texttt{IntersectionNonempty(} $f_{|\N(f)_v},f_{|\N(f)_{-v}}$   \texttt{)} }
\STATE \texttt{v\_connected} $ \leftarrow$  \texttt{CheckConnectivityRecursive(} $f_{|\N(f)_v}$\texttt{)} 
\STATE \texttt{vminus\_connected} $ \leftarrow$  \texttt{CheckConnectivityRecursive(} $f_{|\N(f)_{-v}}$\texttt{)}
\IF{\texttt{v\_connected} \AND  \texttt{vminus\_connected}}
\RETURN \TRUE
\ENDIF
\ENDIF
\ENDFOR

\RETURN \FALSE
\end{algorithmic}
\end{algorithm}

To compute the smallest face $F \subseteq \N(f)$ such that $\sigma_{-}(f) \subseteq F$, one proceeds as follows. First, one finds all the facets of $\N(f)$ (using \texttt{Polymake} \cite{polymake} or \texttt{SageMath} \cite{sagemath}). If $\N(f)$ does not have any facet containing $\sigma_-(f)$, then $\N(f)$ is the smallest face that contains $\sigma_-(f)$. Otherwise, the intersection of the facets containing $\sigma_-(f)$ give the smallest face that contains $\sigma_-(f)$.

One possible way to compute all proper faces $\N(f)_v \subseteq \N(f)$ such that $\sigma(f) \subseteq \N(f)_{v} \cup \N(f)_{-v} $ is the following:
\begin{itemize}
\item[(1)]  Compute the outer normal fan $\mathcal{F}$ of $\N(f)$ and the common refinement $\mathcal{F}  \bigwedge -\mathcal{F}$ of $\mathcal{F}$ and $-\mathcal{F}$ (\cite[Definition 7.6]{Ziegler_book}). Here, $-\mathcal{F}$ is the inner normal fan of $\N(f)$, i.e. the fan obtained by taking the negative of each cone in $\mathcal{F}$. 
\item[(2)] Collect a vector from the relative interior of each cone in $\mathcal{F}  \bigwedge -\mathcal{F}$. These vectors are the normal vectors of the parallel faces of $\N(f)$. 
\item[(3)] Consider all the parallel faces and check whether their union contains $\sigma(f)$. This can be done by taking each vector $v$ from step (2) and computing their scalar product with the exponent vectors. The vector $v$ gives a pair of parallel faces containing $\sigma(f)$ if and only if $\{ v \cdot \mu \mid \mu \in \sigma(f)\}$ has exactly two elements.
\end{itemize}
If the Newton polytope has many faces, computing $\mathcal{F}  \bigwedge -\mathcal{F}$  might become too expensive. In our implementation, we use the following simplification. For a facet $F \subset \N(f)$, there exists a unique (up to scaling) $v \in \mathbb{R}^{n}$ such that $F = \N(f)_v$. We consider all the facets $\N(f)_v$ and check whether $\sigma(f) \subseteq \N(f)_{v} \cup \N(f)_{-v} $. Thus, our code runs through only a sublist of the list in the for-loop in line $8$ in Algorithm \ref{Algo_Conn}. Even under this simplification, the list of facets might become intractably long in practice. Note that there exist $n$-dimensional polytopes with $m$ vertices and
\begin{align*}
\begin{cases}
\vspace{10pt}
\frac{m}{m-\frac{n}{2}}\binom{m-\frac{n}{2}}{m-n} \qquad &\text{if } n  \text{ is even}, \\
2\binom{m-\frac{n+1}{2}}{m-n} \qquad &\text{if } n \text{ is odd,}
\end{cases}
\end{align*}
many facets \cite[Corollary 3.45]{JoswigTheobald_book}.

Using  \texttt{OSCAR} \cite{OSCAR,OSCAR-book} and \texttt{Polymake} \cite{polymake}, we implemented Algorithm \ref{Algo_Conn} in \texttt{Julia}. The code can be found at the Github repository \cite{gitHub}. A rigorous complexity analysis of Algorithm \ref{Algo_Conn} lies outside the scope of the current paper. In the remaining sections, we focus on the practical aspects of Algorithm \ref{Algo_Conn}.

\begin{ex}
\label{Ex_cube4d}
To demonstrate how Algorithm \ref{Algo_Conn} works, consider the polynomial 
\[f(x,y,z,w) = c_1 x+c_2 xy- c_3y-c_4+c_5 yz- c_6z- c_7 xz- c_8 xyz +c_{9}w^3+c_{10}xw,\]
 where $c_1, \dots ,c_{10} \in \mathbb{R}_{>0}$. Since $\# \sigma_-(f) \geq 2, \, \# \sigma_+(f)  \geq 2$ and $\sigma(f)$ does not have a strict separating hyperplane \texttt{CheckConnectivity($f$)} returns false. Following the algorithm, we compute the smallest face $F \subseteq \N(f)$ containing $\sigma_-(f)$. This is a $3$-dimensional face with normal vector $v = (0,0,0,-1)$. Then the algorithm calls \texttt{CheckConnectivityRecursive($f_{|F}$)}, where
\[ f_{|F} = c_1 x+c_2 xy- c_3y-c_4+c_5 yz- c_6z- c_7 xz- c_8 xyz,\]
which is the same polynomial as in Example \ref{Ex_cube}. 

Since $f_{|F}$ has more than one positive and more than one negative exponent vectors and $\sigma(f_{|F})$ does not have a strict separating hyperplane, \texttt{CheckConnectivity($f_{|F}$)} returns false and the algorithm continues with computing the smallest face of $\N(f_{|F}) = F$ containing $\sigma_-(f_{|F})$. Since this smallest face is $F$ itself, the algorithm proceeds with the for-loop in line $8$ of Algorithm \ref{Algo_Conn}.

The faces 
\[G_1 = F_{v_1}, \quad v_1 = (0,0,1,-1), \qquad  G_2 = F_{v_2}, \quad v_2 = (0,0,-1,-1)\]
are parallel and contain all the exponent vectors of $f_{|F}$. As discussed in Example \ref{Ex_cube}, there is an edge with negative vertices between $G_1$ and $G_2$, thus \texttt{IntersectionNonEmpty($f_{|G_1}$,$f_{|G_2}$)} returns true and the method \texttt{CheckConnectivityRecursive} is called again for $f_{|G_1}$ and  $f_{|G_2}$. Since $\sigma(f_{|G_1})$ and $\sigma(f_{|G_2})$ have strict separating hyperplanes, both \texttt{CheckConnectivity($f_{|G_1}$)} and \texttt{CheckConnectivity($f_{|G_2}$)} return true, and the algorithm terminates with the result that $f^{-1}(\mathbb{R}_{<0})$ is connected. The steps taken by the algorithm can be found in Figure \ref{FIG9}.
\end{ex}

\begin{figure}[t]
\centering
\includegraphics[scale=0.75]{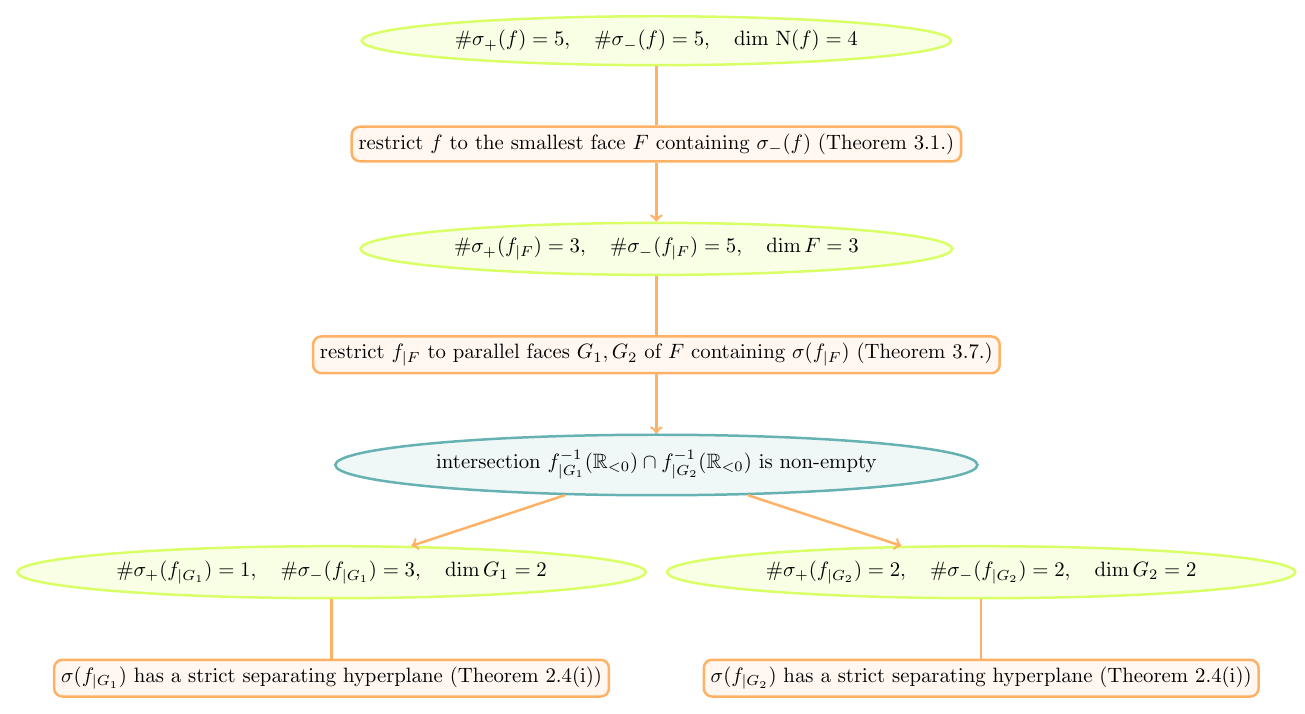}
\caption{{\small Depiction of the steps taken by Algorithm \ref{Algo_Conn} as in Example \ref{Ex_cube4d}. }}\label{FIG9}
\end{figure}

\begin{remark}
Even if Algorithm \ref{Algo_Conn} terminates inconclusively, it can be that $f^{-1}(\mathbb{R}_{<0})$ is connected. To illustrate this, consider the polynomial $f(x) = -x^2 + x -1$, which has one negative connected component. 

Since $f$ has only one variable, more than one negative coefficient, and $\sigma(f)$ does not have a strict separating hyperplane, the conditions in Theorem \ref{Thm::FirstDescartes}(i),(iii),  Corollary~\ref{Cor_OnePos}, Theorem~\ref{Thm::VertCones46}, and in Corollary~\ref{Lemma_SimplexVertCones} are not satisfied. Thus, the submethod \texttt{CheckConnectivity()} returns false. As the Newton polytope does not have any proper face containing $\sigma_-(f)$ and there is no pair of parallel faces containing all the exponent vectors, Algorithm \ref{Algo_Conn} is inconclusive.

It is worth noting that the condition checked by Algorithm \ref{Algo_Conn} is quite restrictive. Specifically, it verifies not only whether the signomial $f(x) = -x^2 + x -1$ has one negative connected component, but also whether, for all polynomials in the family $f_c(x) = -c_2 x^2 + c_1 x -c_0$, where $c_0, c_1, c_2 \in\mathbb{R}_{>0}$, the set $f_c^{-1}(\mathbb{R}_{<0})$ is connected. In the current example, $-x^2 + 3x - 1$ has two negative connected components.
\end{remark}

\subsection{The closure property}
We finish the section with a statement that will allow us to apply Algorithm \ref{Algo_Conn} for reaction networks in Section \ref{Sec_App}. We say that a signomial $f$ satisfies the \emph{closure property} if the closure of $ f^{-1}(\mathbb{R}_{< 0} )$ equals $f^{-1}(\mathbb{R}_{\leq 0})$.

Note that the closure property is not always satisfied. For example, consider $f = x^2 - 2x +1$. While $f^{-1}(\mathbb{R}_{< 0} )$ and consequently its closure are empty,  $f^{-1}(\mathbb{R}_{\leq 0}) = \{ 1 \}$.

\begin{prop}
\label{Prop_ClosureProper}
A signomial $f\colon \mathbb{R}^{n}_{>0} \to \mathbb{R}, \, x \mapsto \sum_{\mu \in \sigma(f)} c_{\mu}x^{\mu}$  satisfies the closure property if one of the following holds:
\begin{itemize}
\item[(i)] $\sigma(f)$ has a strict separating hyperplane.
\item[(ii)]  $\sigma_{-}(f) \subseteq F$, for a proper face $F \subseteq \N(f)$.
\end{itemize}
 \end{prop}

\begin{proof}
(i) has been showed in \cite[Theorem 3.6]{DescartesHypPlane}. (ii) follows by almost the same argument. We recall it for the sake of completeness. Since $f^{-1}(\mathbb{R}_{\leq 0})$ is closed, it contains the closure of $f^{-1}(\mathbb{R}_{< 0})$. To show the reverse inclusion, we pick a point $x \in f^{-1}(\{ 0 \})$ and show that it can be written as a limit of elements from $f^{-1}(\mathbb{R}_{< 0})$.

Let $v \in \mathbb{R}^{n} \setminus \{ 0 \}$ such that $\N(f)_v = F$. Consider the univariate signomial
\[\mathbb{R}_{>0} \to \mathbb{R}, \quad t \mapsto f(t^v \ast x) = \big( \sum_{\mu \in F \cap \sigma(f)} c_{\mu} x^{\mu} \big) t^{d} + \sum_{\mu \in \sigma(f) \setminus F} c_{\mu} x^{\mu} t^{v \cdot \mu},\]
where $d=\max_{\mu \in \sigma(f)} v \cdot \mu$. The leading coefficient of $f(t^v \ast x)$ is negative since otherwise $f(x) > 0$. Since $F$ is a proper face, $\sigma(f) \setminus F \neq \emptyset$, therefore the trailing coefficient of $f(t^v \ast x)$ is positive. By Descartes' rule of signs, $1$ is a positive real root of $f(t^v \ast x)$ and $(1+\tfrac{1}{n})^v\ast x \in f^{-1}(\mathbb{R}_{<0})$ for all $n \in \mathbb{N}$. Furthermore, the sequence $\{(1+\tfrac{1}{n})^v\ast x \}_{n \in \mathbb{N}}$ converges to $x$.
\end{proof}

\section{Application to reaction networks}
\label{Sec_App}
A reaction network is a collection of reactions between (bio)chemical species. The species are represented by formal variables and the reactions by arrows between non-negative integer linear combination of the species. The goal of reaction network theory is to study the evolution of the concentration of the species in time, which is usually modeled by an ordinary differential equation system (ODE). For an introduction to reaction network theory, we refer to \cite{CRN_Dickenstein}. Here, we consider some specific reaction networks, where Algorithm \ref{Algo_Conn} can be applied to verify that the \emph{parameter region of multistationarity} is path connected.

\begin{remark}
An open subset of a Euclidean space is connected if and only if it is path connected. Thus, for negative connected components of a signomial these two notions of connectivity coincide. As the parameter region of multistationarity might not be an open set, path-connectivity is a stronger property than connectivity.
\end{remark}

\subsection{Weakly irreversible phosphorylation cycle}
First, we consider the reaction network that was studied in \cite{ParamGeo}:
\begin{equation}
\label{Eq_weaklyr}
\begin{aligned}
&\mathrm{S} +\mathrm{E} \xrightleftharpoons[\kappa_{2}]{\kappa_{1}} Y_1 \xrightarrow{\kappa_{3}} Y_2 \xrightleftharpoons[\kappa_{5}]{\kappa_{4}} \mathrm{S}_{\rm p} + \mathrm{E} \xrightleftharpoons[\kappa_{7}]{\kappa_{6}} Y_3 \xrightarrow{\kappa_{8}} Y_4 \xrightleftharpoons[\kappa_{10}]{\kappa_{9}} \mathrm{S}_{{\rm p}{\rm p}} + \mathrm{E} \\
&\mathrm{S}_{{\rm p}{\rm p}} +\mathrm{F} \xrightleftharpoons[\kappa_{12}]{\kappa_{11}} \mathrm{Y}_5 \xrightarrow{\kappa_{13}} \mathrm{Y}_6 \xrightleftharpoons[\kappa_{15}]{\kappa_{14}} \mathrm{S}_{\rm p} + \mathrm{F} \xrightleftharpoons[\kappa_{17}]{\kappa_{16}} \mathrm{Y}_7 \xrightarrow{\kappa_{18}} \mathrm{Y}_8 \xrightleftharpoons[\kappa_{20}]{\kappa_{19}} \mathrm{S} + \mathrm{F}, 
\end{aligned}
\end{equation}
which represents the two-site phosphorylation cycle of a substrate $\mathrm{S}$, where both the phosphorylation and dephosphorylation processes follow a weakly irreversible mechanism. The reaction network has $13$ species and $20$ reactions. Under the assumption of mass action kinetics, the evolution of the concentration of the species is modeled by an ODE system of the form:
\begin{align*}
\dot{x} = f_\kappa(x),
\end{align*}
where $\kappa = (\kappa_1,\dots, \kappa_{20})$ are positive parameters, called \emph{reaction rate constants}, and $f_\kappa$ is a polynomial in $13$ variables. One can find the exact form of $f_\kappa$ in the accompanying \texttt{Jupyter} notebook \cite{gitHub}.

The set of all positive steady states of the ODE  system equals the variety
\[V_\kappa := \{ x \in \mathbb{R}^{13}_{>0} \mid f_\kappa(x) = 0\}.\]
For a given initial condition, the trajectories of the ODE are contained in \emph{stoichiometric compatibility classes} that are defined as
\[\mathcal{P}_c := \{ x \in \mathbb{R}^{13}_{>0} \mid Wx = c\},\]
where $c \in \mathbb{R}^{3}$ is the so-called \emph{total concentration parameter} and $W \in  \mathbb{R}^{3 \times 13}$ is a matrix whose rows give the \emph{conservation laws} of the network. A steady state $x \in V_\kappa \cap \mathcal{P}_c$ is called a \emph{relevant boundary steady state} if some of the coordinates of $x$ are zero and $\mathcal{P}_c \cap \mathbb{R}^n_{>0} \neq \emptyset$. A pair of parameters $(\kappa,c)$ \emph{enables multistationarity} if $V_\kappa \cap \mathcal{P}_c$ contains at least two points. The \emph{parameter region of multistationarity} is the set of all pairs $(\kappa,c)$ enabling multistationarity.

By sampling parameters and using a connectivity graph, connectivity of the parameter region of multistationarity  had been studied for the weakly irreversible phosphorylation cycle~\eqref{Eq_weaklyr} in~\cite{ParamGeo}. Based on that numerical observation, it was conjectured that the region is connected.

 In \cite[Algorithm 2.5]{EF_connpaper}, the authors gave an algorithm that certifies that the parameter region of multistationarity is path connected for reaction networks satisfying some technical conditions. The algorithm is based on the following result.
\begin{prop} 
\label{Lemma_CRNconn}
\cite[Theorem 2.4]{EF_connpaper} For a conservative reaction network without relevant boundary steady states, there exists a polynomial
\[q\colon \mathbb{R}^{n+ \ell}_{>0} \to \mathbb{R}\]
such that if $q$ satisfies the closure property and $q^{-1}(\mathbb{R}_{<0})$ is path connected, then the parameter region of multistationarity is path connected.
\end{prop}

In \cite{FeliuPlos,EF_connpaper}, it was described how to check whether the network is conservative, does not have any relevant boundary states and how to compute the polynomial $q$. To check if $q$ satisfies the closure property and $q^{-1}(\mathbb{R}_{<0})$ is path connected, the authors in \cite{EF_connpaper} used strict separating hyperplanes \cite[Theorem 3.6]{DescartesHypPlane}. Using this strategy, it was verified for a large number of reaction networks that the parameter region of multistationarity is path connected. 

The weakly irreversible phosphorylation cycle \eqref{Eq_weaklyr} satisfies the technical conditions of Proposition \ref{Lemma_CRNconn}, i.e. it is conservative and does not have any relevant boundary steady states. The polynomial $q$ from Proposition \ref{Lemma_CRNconn} associated with the network \eqref{Eq_weaklyr} is a large polynomial with $1248$ monomials.

The method in \cite[Algorithm 2.5]{EF_connpaper} was inconclusive, because $\sigma(q)$ does not have a strict separating hyperplane. In the following, we go through Algorithm \ref{Algo_Conn} step-by-step and show that $q$ satisfies the closure property and $q^{-1}(\mathbb{R}_{<0})$ is path connected, which implies by Proposition \ref{Lemma_CRNconn} that the parameter region of multistationarity for the weakly irreversible phosphorylation cycle \eqref{Eq_weaklyr} is path connected. The computations were done using \texttt{OSCAR} \cite{OSCAR,OSCAR-book} and \texttt{Polymake} \cite{polymake}. The code can be found at the Github repository \cite{gitHub}.

The Newton polytope of $q$ has dimension $16$, $1020$ out of its $1248$ monomials are positive, and $228$ are negative. The smallest face $F \subseteq \N(q)$ that contains $\sigma_-(q)$ has dimension $12$ and contains $212$ positive exponent vectors of $q$. Using Proposition \ref{Prop_ClosureProper}, we conclude that $q$ satisfies the closure property.

Following Algorithm \ref{Algo_Conn}, we study the restricted polynomial $q_{|F}$. The Newton polytope $\N(q_{|F})$ has $37$ facets. We choose the first facet $G_1 = \N(q)_v$ that is provided by \texttt{Polymake}. For the face $G_2 = \N(f)_{-v}$ it holds that $\sigma(q_{|F}) \subseteq G_1 \cup G_2$. Furthermore, there exists a pair of negative exponent vectors $\beta_1, \beta_2$ such that $\beta_1 \in G_1$ and $\beta_2 \in G_2$ and $\Conv(\beta_1, \beta_2)$ is an edge of $\N(q_{|F})$. Thus, by Theorem \ref{Thm_ParallelFacesEdge}, it is enough to show that $q_{|G_1}^{-1}(\mathbb{R}_{<0})$ and $q_{|G_2}^{-1}(\mathbb{R}_{<0})$ are connected. This holds by Theorem \ref{Thm::FirstDescartes}(i), since the support of $q_{|G_{1}}$ and $q_{|G_{2}}$ have strict separating hyperplanes. For an overview of the steps taken by Algorithm \ref{Algo_Conn}, we refer to Figure \ref{FIG_Gunawardena}. 

\begin{thm}
The parameter region of multistationarity of the weakly irreversible phosphorylation system with two binding sites \eqref{Eq_weaklyr} is path connected.
\end{thm}

\begin{figure}[t]
\centering
\includegraphics[scale=0.7]{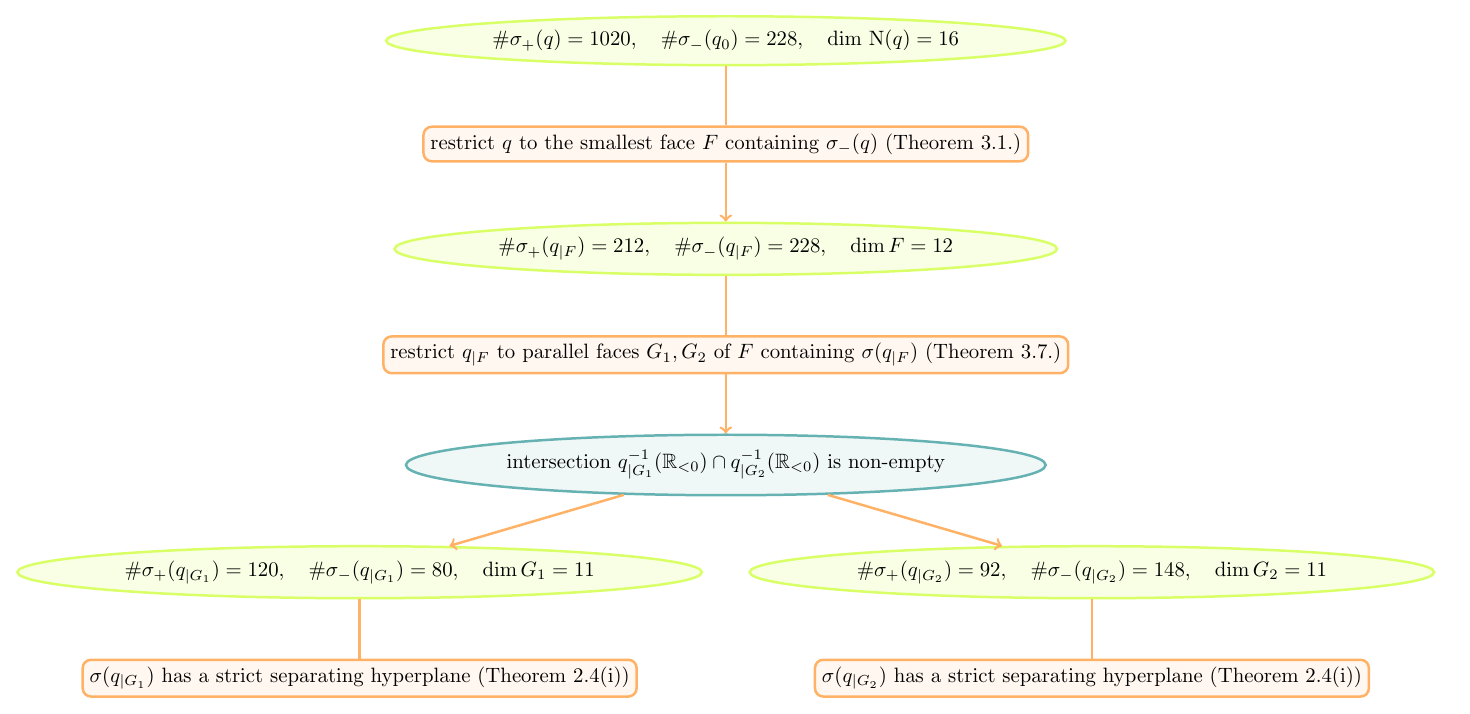}
\caption{{\small Steps taken by Algorithm \ref{Algo_Conn} with input polynomial $q$ from Proposition \ref{Lemma_CRNconn} associated with the weakly irreversible phosphorylation cycle \eqref{Eq_weaklyr}.}}\label{FIG_Gunawardena}
\end{figure}

\medskip

\subsection{Strongly irreversible phosphorylation cycles} A well-studied family of reaction networks is the family of $m$-site phosphorylation cycles, where it is assumed that each phosphorylation step follows the Michaelis-Menten mechanism. This is also referred to as \emph{strong irreversibility} in \cite{MultiSite_Phosph_Gunaw}. For a fixed $m \in \mathbb{N}$, the $m$-site phosphorylation cycle is given by the reactions:
\begin{align*}
\mathrm{S}_{i} +\mathrm{E} \xrightleftharpoons[\kappa_{6i+2}]{\kappa_{6i+1}} \mathrm{ES}_i \xrightarrow{\kappa_{6i+3}} \mathrm{S}_{i+1} + \mathrm{E}, \quad \mathrm{S}_{i+1} +\mathrm{F} \xrightleftharpoons[\kappa_{6i+5}]{\kappa_{6i+4}} \mathrm{FS}_{i+1} \xrightarrow{\kappa_{6i+6}} \mathrm{S}_{i+6} + \mathrm{F}, \quad i =0,\dots,n-1.
\end{align*}

Questions about the existence of multistationarity and the number of steady states are well understood \cite{MultiSite_Phosph_Sontag,G-distributivity,markev,MultiSite_Phosph_Dicken,rendall_feliu}. It is known that the projection of the parameter region of multistationarity onto the space of reaction rate constants (the $\kappa$'s) is path connected for all $m \geq 2$ \cite{Multnsite}. In \cite{EF_connpaper}, it was shown that the full parameter region (the region in the $\kappa$'s and $c$'s) of multistationarity is path connected for $m=2,3$, which led to the conjecture that this might be true for every $m \geq 2$.

Let $q_m$ denote the polynomial from Proposition \ref{Lemma_CRNconn} associated with the $m$-site phosphorylation cycle. Algorithm \ref{Algo_Conn} can be successfully applied to show that $q_{m}^{-1}(\mathbb{R}_{<0})$ is path connected for $m = 4,5,6,7$. The computation can be found in the accompanying \texttt{Jupyter} notebook \cite{gitHub}. Since in these cases, the negative monomials are contained in a proper face of $\N(q_m)$, the polynomial $q_m$ satisfies the closure property by Proposition \ref{Prop_ClosureProper}. We conclude that the parameter region of multistationarity for the $m$-site phosphorylation cycle is path connected for $m =4,5,6,7$.

The steps taken by Algorithm \ref{Algo_Conn} are similar for each $m=4,5,6,7$. This observation led us to a proof strategy for the conjecture that the parameter region of multistationarity for both the weakly and the strongly irreversible phosphorylation network is connected for all $m \geq 2$. This will be the content of the upcoming paper \cite{NsiteConnecGeneral}.

\section*{Acknowledgments}
The author thanks the anonymous reviewer, whose comment led to an improvement of Proposition \ref{Thm_ParallelFaces} and shortening its proof. Additionally, the author is thankful to Elisenda Feliu and Nidhi Kaihnsa for useful discussions and comments on the manuscript. Funded by the European Union under the Grant Agreement no. 101044561, POSALG. Views and opinions expressed are those of the author(s) only and do not necessarily reflect those of the European Union or European Research Council (ERC). Neither the European Union nor ERC can be held responsible for them.

{\small


}

\end{document}